\documentclass[11pt, reqno]{amsart}
\usepackage{epsfig}
\usepackage{amsmath, amsthm, amssymb, mathrsfs, bbm}

\usepackage{graphicx,xypic}
\usepackage{fullpage}
\numberwithin{equation}{section}
\newcommand{\heis}{\ensuremath{\mathbf{H}}}
\newcommand{\surf}[1][g]{\ensuremath{\Sigma_{#1}}}
\newcommand{\hbdy}[1][g]{\ensuremath{\mathrm{H}_{#1}}}
\newcommand{\esurf}[1][g]{\ensuremath{\mathbf{\Sigma}_{#1}}}
\newcommand{\surfcyl}[1][g]{\ensuremath{\surf[#1]\times [0,1]}}
\newcommand{\mcg}[1]{\ensuremath{\mbox{MCG}(#1)}}
\newcommand{\emcg}[1]{\ensuremath{\widetilde{\mbox{MCG}}(#1)}}
\newcommand{\thetafn}{\ensuremath{\mathbf{\Theta}_N^\Pi}}
\newcommand{\skn}[1]{\ensuremath{< \! #1 \! >}}
\newcommand{\lskein}[1]{\ensuremath{\mathcal{L}(#1)}}
\newcommand{\rlskein}[1]{\ensuremath{{\mathcal{L}}_N(#1)}}

\newcommand{\mapcyl}[1]{\ensuremath{\mathrm{I}_{#1}}}
\newcommand{\anomaly}{\ensuremath{a}}
\newcommand{\innprod}{\ensuremath{\langle -,- \rangle}}
\newcommand{\dilim}[2]{\ensuremath{\varinjlim_{#1} #2}}
\DeclareMathOperator{\id}{id}
\DeclareMathOperator{\Dim}{Dim}
\DeclareMathOperator{\op}{Op}
\DeclareMathOperator{\Aut}{Aut}
\DeclareMathOperator{\End}{End}
\DeclareMathOperator{\Int}{Int}
\newtheorem{theorem}{Theorem}[section]

\newtheorem{proposition}[theorem]{Proposition}

\theoremstyle{definition}
\newtheorem{definition}[theorem]{Definition}

\theoremstyle{remark}
\newtheorem{remark}[theorem]{Remark}

\begin{document}

\title{The topological quantum field theory of Riemann's theta functions}
\author{R{\u{a}}zvan Gelca}
\address{Department of Mathematics and Statistics, Texas Tech University, Lubbock, TX 79409-1042. USA.} \email{rgelca@gmail.com}
\thanks{Research of the first author partially supported by the NSF, award No. DMS 0604694}
\author{Alastair Hamilton}
\address{Department of Mathematics and Statistics, Texas Tech University, Lubbock, TX 79409-1042. USA.} \email{alastair.hamilton@ttu.edu}
\thanks{Research of the second author partially supported by  the Simons Foundation, award No. 279462}

\begin{abstract}
In this paper we prove the existence and uniqueness of a topological quantum field theory that incorporates, for all Riemann surfaces, the corresponding spaces of theta functions and the actions of the Heisenberg groups and modular groups on them.
\end{abstract}

\maketitle

\section{Introduction}

In his treatise ``\emph{Theorie del Abel'schen Funktionen}'' \cite{riemann}, Bernard Riemann associated Riemann surfaces to  elliptic functions, and to these Riemann surfaces he associated theta functions defined on $n$-tori which we now call Jacobian varieties. Riemann's work, inspired by the previous works of Abel and Jacobi, generalized their ideas. Theta functions were originally the building blocks of the theory of elliptic functions, and later established themselves as some of the most important functions in mathematics, comparable to trigonometric functions and polynomials. A major contribution of 19th century research was the discovery of the action of mapping class groups of surfaces (also known as modular groups) on theta functions, whose discovery is mostly due to Jacobi. Much later, Andr\'e Weil  discovered an action of a Heisenberg group on theta functions given by translations in the variables \cite{weil}.

At the end of the 20th century, Riemann's theta functions were placed in a quantum physical framework. On the one hand, Edward Witten related them to the abelian version of his quantum field theory based on the Chern-Simons functional. On the other hand; theta functions, the action of the Heisenberg group, and the action of the modular group were obtained from the geometric quantization of the Jacobian variety, cf. \cite{manin1}.

This paper is about the theory of  Riemann's theta functions and its place within Witten's abelian Chern-Simons theory. It is a continuation of the work of the first author with Alejandro Uribe in \cite{thetaTQFT}, as well as of the work of Murakami, Ohtsuki, and Okada in \cite{moo}. The main result of this paper is the construction of the topological quantum field theory (TQFT) that underlies the theory of Riemann's theta functions and is hence associated to abelian Chern-Simons theory. This TQFT also encompasses the Murakami-Ohtsuki-Okada invariant of 3-manifolds.

The construction is done in the perspective of \cite{bhmv} using skein modules associated to the linking number \cite{przytycki2}. Our perspective is novel in the sense that we view this topological quantum field theory as an object that incorporates
 the spaces of theta functions, the actions of the
Heisenberg groups and modular groups for surfaces of all genera.
Stated more precisely, this TQFT is constructed so that it incorporates the skein theoretical realizations of these
three objects. Moreover, we prove the \emph{uniqueness} of our description of such a theory, which means that there is only one theory that incorporates the classical theory of theta functions, namely abelian Chern-Simons theory. From this viewpoint, Chern-Simons theory arises naturally from the theory of theta functions. In fact, as is well-known, many of the classical constructions in the theory of theta functions can be done from the perspective of  quantization. As we will see, these principles of  quantization  largely fix what the corresponding TQFT must look like. We hope that this point of view may be useful in studying TQFTs based on other gauge groups.

It is well-known that modular tensor categories give rise to TQFTs \cite{turaev}, although it is an open question as to whether all TQFTs arise in this way. One can pose the question as to whether or not the TQFT defined in this paper arises from some modular category? Unfortunately, there are some issues with constructing such a modular category. Even if such a modular category exists, it would be impossible to embed this tensor category inside the tensor category of vector spaces, at least when the latter is given the usual associator for the tensor product. In some sense, any such category must look somewhat pathological. We briefly discuss some of the issues that arise at the end of the paper, following the proof of the main theorem.

The paper is organized as follows. In Section \ref{sec_thetaskein} we recall relevant material on theta functions, skein theory and the action of the mapping class group on theta functions. In Section \ref{sec_resolve} we resolve the projective ambiguity in the representation of the mapping class group using the Maslov index. In Section \ref{sec_tqft} we define the background for our TQFT and prove the main theorem on the existence and uniqueness of a TQFT incorporating the theory of theta functions. Here we construct a certain category of extended surfaces, where composition of morphisms is defined through the Maslov index. This modest innovation reduces keeping track of the anomaly of a TQFT and related matters to a series of category theoretical tautologies.

\subsection*{Notation and conventions}

Given an oriented surface $\Sigma$, we will denote its mapping class group by $\mcg{\Sigma}$. We denote the group of integers modulo $N$ by $\mathbb{Z}_N$.

\section{Theta functions and skein theory} \label{sec_thetaskein}

In this section we will recall basic background material on theta functions and quantization, translating the relevant notions into skein theory.

\subsection{Theta functions and quantization}

Let $\surf$ denote the standard surface of genus $g$, as depicted in Figure~\ref{fig_surface}, \emph{which we fix for the remainder of the paper}. Figure~\ref{fig_surface} illustrates a canonical basis for $H_1(\surf,\mathbb{Z})$, which we also fix. $\surf$ is oriented so that the intersection form in $H_1(\surf,\mathbb{Z})$ is given by,
\begin{eqnarray*}
a_i \cdot a_j = 0 = b_i \cdot b_j, \quad
a_i \cdot b_j = \delta_{ij}.
\end{eqnarray*}

\begin{figure}[ht]
\centering
\resizebox{.40\textwidth}{!}{\includegraphics{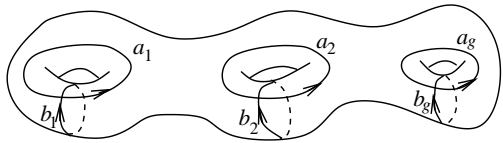}}
\caption{The standard surface of genus $g$ together with a canonical basis of its first homology group.}
\label{fig_surface}
\end{figure}

Endow $\surf$ with a complex structure to turn it into a Riemann surface. We now briefly sketch the standard construction of the space of theta functions using this complex structure via geometric quantization. Later, we will see that this construction does not really depend upon this choice of complex structure by describing a topological model for the space of theta functions in terms of skein theory; however, it is this complex structure that allows us to make use of geometric quantization.

Recall from \cite{farkaskra} that given the complex structure on $\surf$, there exists a unique set of holomorphic 1-forms
$ \zeta_1,\ldots,\zeta_g $
on $\surf$ such that
\[ \int_{a_i}\zeta_j = \delta_{ij}. \]
From these 1-forms we  define the matrix $\Pi\in M_g(\mathbb{C})$ by
\[ \pi_{ij}:=\int_{b_i}\zeta_j. \]
This matrix is symmetric with positive definite imaginary part. The $g\times 2g$ matrix $\lambda :=(I_g,\Pi)$ is called the period matrix.

\begin{definition}
The Jacobian variety $\mathcal{J}$ associated to our Riemann surface is the quotient
\[ \mathcal{J}:=\mathbb{C}^g/L \]
of $\mathbb{C}^g$ by the lattice subgroup
\[ L:=\{ t_1\lambda_1+\cdots+t_{2g}\lambda_{2g}: t_1,\ldots,t_{2g}\in\mathbb{Z} \} \] spanned by the columns $\lambda_i$ of $\lambda$. It inherits a holomorphic structure from its covering space $\mathbb{C}^g$.
\end{definition}

The period matrix $\lambda$ defines an invertible $\mathbb{R}$-linear map
\begin{eqnarray*}
\mathbb{R}^{2g}=\mathbb{R}^g\times\mathbb{R}^g  \to  \mathbb{C}^g \quad
(x,y)  \mapsto  \lambda (x, y)^T = x + \Pi y
\end{eqnarray*}
descending to a diffeomorphism
\[ \mathbb{R}^{2g}/\mathbb{Z}^{2g}\approx\mathbb{C}^g/L=\mathcal{J}. \]
In this system of real coordinates we may unambiguously define a symplectic form
\[ \omega:=\sum_{i=1}^g dx_i\wedge dy_i \]
on $\mathcal{J}$. This symplectic form allows us to view $\mathcal{J}$ as the phase space of a classical mechanical system.

The Hilbert space of our quantum theory is obtained by applying the technique of geometric quantization (cf. \cite{sniatycki}, \cite{woodhouse}) to this classical system using a K\"ahler polarization, which we will describe here very briefly. One begins by constructing a certain holomorphic line bundle over the Jacobian variety known as the prequantization line bundle. This bundle is constructed as the tensor product of a line bundle that possesses a connection whose curvature is equal to $-\frac{2\pi i}{h}\omega$ and a line bundle that is the square root of
the canonical line bundle. In our particular case the latter is trivial, so we are only concerned with the first.
Because Chern classes are integral classes, this forces Planck's constant $h$ to be the reciprocal of an integer $N$. In fact, we insist that this is a positive \emph{even} integer, in order for the entire mapping class group to act on theta functions. We then consider sections of this line bundle. In the K\"ahler polarization, our Hilbert space is defined to be the space of holomorphic sections of this line bundle. Pulling this line bundle back to the contractible space $\mathbb{C}^g$, it becomes trivial, and the Hilbert space may be identified with holomorphic functions on $\mathbb{C}^g$ satisfying certain periodicity conditions. In a canonical trivialization that is determined by the construction of this bundle, we get the space of theta functions, which is defined as follows.

\begin{definition}
The space of classical theta functions $\thetafn(\surf)$ associated to our Riemann surface consists of all holomorphic functions $f:\mathbb{C}^g\to\mathbb{C}$ satisfying the periodicity conditions
\begin{displaymath}
f(z+\lambda_j) = f(z), \quad
f(z+\lambda_{g+j}) = e^{-2\pi iNz_j-\pi i N\pi_{jj}}f(z);
\end{displaymath}
for $j=1,\ldots, n$; where $\lambda_j$ denotes the $j$th column of the period matrix $\lambda$. It may be given the structure of a Hilbert space by defining on it the following Hermitian inner product,
\begin{equation} \label{eqn_thetaform}
\langle\langle f,g \rangle\rangle := (2N)^{\frac{g}{2}} (\det\Pi_{\mathrm{I}})^{\frac{1}{2}} \int_{[0,1]^{2g}} f(x,y)\overline{g(x,y)} e^{-2\pi N y^T \Pi_{\mathrm{I}} y} dx dy;
\end{equation}
where $\Pi_\mathrm{I}\in M_g(\mathbb{R})$ denotes the imaginary part of $\Pi$.
\end{definition}

The space of theta functions, as defined above, has a canonical orthonormal basis consisting of the \emph{theta series}
\[ \theta_\mu^{\Pi}(z):=\sum_{n\in {\mathbb Z}^g}e^{2\pi i N\left[ \frac{1}{2}\left(\frac{\mu}{N}+n\right)^T \Pi \left(\frac{\mu}{N}+n\right)+\left(\frac{\mu}{N}+n\right)^T z \right]}, \quad \mu\in \mathbb{Z}_N^g.\]

Now we turn to the problem of quantization, for which we use Weyl quantization in the momentum representation \cite{andersen}, \cite{thetaTQFT}. In this scheme, only the exponential functions
\begin{eqnarray*}
\mathcal{J}  \to  \mathbb{C}, \quad
(x,y) \mapsto  e^{2\pi i (p^T x + q^T y)}
\end{eqnarray*}
on the phase space $\mathcal{J}$ are quantized, where $p,q\in\mathbb{Z}^g$. To such a function on the phase space, one defines an operator
\begin{equation} \label{eqn_quantop}
O_{pq}:= \op\left(e^{2\pi i (p^T x + q^T y)}\right):\thetafn(\surf) \to \thetafn(\surf)
\end{equation}
on the Hilbert Space $\thetafn(\surf)$ by exponentiating the action of the usual position and momentum operators. On theta series it is given by (cf. Proposition 2.1 of \cite{thetaTQFT})
\[ O_{pq}[\theta_\mu^\Pi] = e^{-\frac{\pi i}{N}p^Tq-\frac{2\pi i}{N}\mu^Tq}\theta_{\mu+p}^\Pi. \]

\subsection{Heisenberg group}

\begin{definition}
Let $g$ be a nonnegative integer. The \emph{Heisenberg group} $\heis(\mathbb{Z}^g)$ is the group
\[ \heis(\mathbb{Z}^g):=\{(p,q,k): \; p,q\in\mathbb{Z}^g, k\in \mathbb{Z}\} \]
with underlying multiplication
\[ (p,q,k)(p',q',k')=\left(p+p',q+q',k+k'+\sum_{j=1}^g(p_jq_j'-p_j'q_j)\right).\]
If $N$ is a positive \emph{even} integer, the \emph{finite Heisenberg group} $\heis(\mathbb{Z}_N^g)$ is the quotient of the group $\heis(\mathbb{Z}^g)$ by the normal subgroup consisting of all elements of the form
\[ (p,q,2k)^N=(Np,Nq,2Nk); \quad p,q\in\mathbb{Z}^g, k\in\mathbb{Z}. \]
\end{definition}

\begin{remark} \label{rem_heisenberghomology}
The Heisenberg group $\heis(\mathbb{Z}^g)$ can be interpreted as the $\mathbb{Z}$-extension of
\begin{displaymath}
\begin{array}{ccc}
H_1(\surf,\mathbb{Z}) & = & \mathbb{Z}^g \times \mathbb{Z}^g, \\
\sum_{i=1}^g (p_i a_i + q_i b_i) & \leftrightharpoons & (p,q);
\end{array}
\end{displaymath}
by the cocycle defined by the intersection form.
\end{remark}

The operators \eqref{eqn_quantop} defined by Weyl quantization generate a subgroup of the group $U(\thetafn(\surf))$ of unitary operators on $\thetafn(\surf)$, which may be identified with the finite Heisenberg group as follows.

\begin{proposition} (Proposition 2.3 in \cite{thetaTQFT})
The subgroup $G$ of the group of unitary operators on $\thetafn(\Sigma_g)$ that is generated by all operators of the form
\[ O_{pq}=\op\left(e^{2\pi i (p^T x + q^T y)}\right);\quad p,q\in\mathbb{Z}^g; \]
is isomorphic to the finite Heisenberg group $\heis(\mathbb{Z}_N^g)$:
\begin{equation} \label{eqn_schrorep}
\heis(\mathbb{Z}_N^g) \cong  G, \quad
(p,q,k)  \mapsto  e^{\frac{k\pi i}{N}} O_{pq}.
\end{equation}
\end{proposition}

The representation of the finite Heisenberg group on the space of theta functions that is defined by \eqref{eqn_schrorep} is called the \emph{Schr\"odinger representation}. It is essentially unique, as expressed by the following well-known finite dimensional version of the Stone-von Neumann theorem (see \cite{thetaTQFT}).

\begin{theorem} \label{thm_stone}
Any irreducible representation of the finite Heisenberg group in which the element $(0,0,1)\in\heis(\mathbb{Z}_N^g)$ acts as multiplication by the scalar $e^{\frac{\pi i}{N}}$ is unitarily equivalent to the Schr\"odinger representation \eqref{eqn_schrorep}.
\end{theorem}

The Stone-von Neumann theorem allows us to quantize changes of coordinates, giving rise to a projective representation of the mapping class group $\mcg{\surf}$ of our surface on $\thetafn(\surf)$. Any element $h\in\mcg{\surf}$  induces a linear endomorphism $h_*$ of $H_1(\surf,\mathbb{Z})$. By Remark \ref{rem_heisenberghomology}, this in turn induces an automorphism
\begin{eqnarray*}
\heis(\mathbb{Z}_N^g)  \to  \heis(\mathbb{Z}_N^g),\quad
((p,q),k)  \mapsto  (h_*(p,q),k)
\end{eqnarray*}
of the finite Heisenberg group. By Theorem \ref{thm_stone}, there is a unitary map
\[ \rho(h):\thetafn(\surf)\to\thetafn(\surf) \]
such that
\begin{equation} \label{eqn_egorov}
O_{h_*(p,q)}\circ\rho(h) =\rho(h)\circ O_{pq};\quad p,q\in\mathbb{Z}^g.
\end{equation}
This identity is referred to as the \emph{exact Egorov identity}. By Schur's Lemma, the map  $\rho(h)$ is well-defined up to multiplication by a complex scalar of unit modulus. Consequently, this construction gives rise to a projective unitary representation
\begin{equation} \label{eqn_hermitejacobiact}
\mcg{\surf}  \to  U\left(\thetafn(\Sigma_g)\right)/U(1), \quad
h  \mapsto  \rho(h);
\end{equation}
of the mapping class group on the space of theta functions. There is a well-known representation theoretical point of view which identifies the operators associated to elements of the mapping class group as discrete Fourier transforms, cf. \cite{thetaTQFT}.

\begin{remark}
The classical way of explicating the representation of the mapping class
group, which was done in the 19th century well before Weil's
discovery of the action of the Heisenberg group, is to
examine the behavior of theta functions under the action
of the mapping class group on the variables of theta functions.
Under changes of coordinates, up to a multiplication by a holomorphic function, the space
of theta functions is mapped linearly to itself (see for example \cite{mum}). Because changes of coordinates act on the
symbols of the operators that make up the Heisenberg group in
exactly the same manner as the one described above, the two representations of
the mapping class group coincide up to multiplication by a
scalar. The above point of view allows us to establish the
analogy with the metaplectic representation.  However,
in what follows we will give a topological perspective of the representation of the mapping class group, and in this setting
we  outline a topological approach to resolving the
projective ambiguity of the representation, different from
the classical one.
\end{remark}

\subsection{Skein theory} \label{sec_skeinthy}

We  now recall from \cite{thetaTQFT} a topological model for the space of theta functions and of the actions of the Heisenberg group and the mapping class group in terms of skein theory. Given a compact oriented 3-manifold $M$, consider the free $\mathbb{C}[t,t^{-1}]$-module with basis  the set of isotopy classes of framed oriented links in the interior of the manifold, including the empty link $\emptyset$. Factor this by the $\mathbb{C}[t,t^{-1}]$-submodule generated by the relations from Figure~\ref{fig_skeinrelations}, where the two terms depict framed links that are identical, except in an embedded ball, in which they appear as shown and are endowed with the blackboard framing; the orientation in the embedded ball shown in Figure \ref{fig_skeinrelations} that is induced by the orientation of $M$  coincides with the canonical orientation of $\mathbb{R}^3$.

\begin{figure}[ht]
\centering
\includegraphics{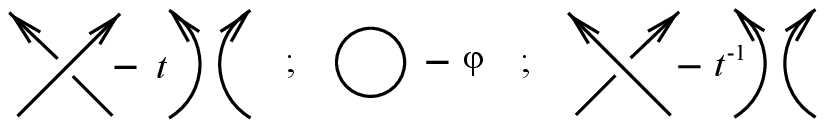}
\caption{The skein relations defining $\lskein{M}$.}
\label{fig_skeinrelations}
\end{figure}

The result of this factorization is called the \emph{linking number skein module} of $M$ and is denoted by $\lskein{M}$. We refer to its elements as skeins. Given a link $L$ in $M$, we denote its skein by $\skn{L}\in\lskein{M}$. From the positive even integer $N$, we define the \emph{reduced linking number skein module} of $M$, denoted by $\rlskein{M}$, as the $\mathbb{C}[t,t^{-1}]$-module obtained from $\lskein{M}$ as the quotient by the following relations
\[ t\cdot L= e^{\frac{\pi i}{N}}\cdot L \quad \text{and} \quad \gamma^{\| N}\cup L=L \]
where the first relation holds for all links $L$, and the second relation holds for all oriented framed simple closed curves $\gamma$ and links $L$ disjoint from $\gamma$, and $\gamma^{\| N}$ is the multicurve in a regular neighborhood of $\gamma$ disjoint from $L$ obtained by replacing $\gamma$ by $N$ parallel copies of it (where ``parallel'' is defined using the framing of $\gamma$).

If $M=\surfcyl$, then the identification
\[ (\surfcyl)\bigsqcup_{\surf\times\{0\}  =\surf\times\{1\} }(\surfcyl) \approx \surfcyl \]
induces a multiplication of skeins in $\rlskein{\surfcyl}$ which transforms it into an algebra. Note that a $\mathbb{C}$-linear basis for $\rlskein{\surfcyl}$ is given by
\[ a_1^{p_1}\cdots a_g^{p_g}b_1^{q_1}\cdots b_g^{q_g}; \qquad p,q\in\mathbb{Z}_N^g. \]
where $a_i$ and $b_i$ are the simple curves depicted in Figure \ref{fig_surface} (cf. Theorem 4.7(a) of \cite{thetaTQFT}).

Furthermore, for any manifold $M$, the identification
\begin{equation} \label{eqn_collarid}
(\partial M\times [0,1])\bigsqcup_{ \partial M\times\{0\}  =\partial M } M\approx M
\end{equation}
defined canonically up to isotopy makes $\rlskein{M}$ into a $\rlskein{\partial M\times [0,1]}$-module.

\begin{theorem} (Theorem 4.7(c) in \cite{thetaTQFT})\label{thm_opskein}
There is an isomorphism of algebras
\begin{equation} \label{eqn_opskein}
\rlskein{\surfcyl}  \cong  \End(\thetafn(\Sigma_g)), \quad
a_1^{p_1}\cdots a_g^{p_g}b_1^{q_1}\cdots b_g^{q_g}  \mapsto  e^{\frac{\pi i}{N}p^Tq}O_{pq}.
\end{equation}
This isomorphism is equivariant with respect to the action of the mapping class group $\mcg{\surf}$, which acts on $\rlskein{\surfcyl}$ in the obvious fashion and on $\End(\thetafn(\Sigma_g))$ through conjugation by the  action of the mapping class group;
\begin{displaymath}
\mcg{\surf}\times\End(\thetafn(\Sigma_g))  \to  \End(\thetafn(\Sigma_g)), \quad
(h,A)  \mapsto  \rho(h)\circ A\circ\rho(h)^{-1}.
\end{displaymath}
\end{theorem}

We have the following  topological description of the space of theta functions. Let $\hbdy$ denote the standard genus $g$ handlebody, which we define as the 3-manifold enclosed by the surface $\surf$ from Figure~\ref{fig_surface}. \emph{$\hbdy$ will remain fixed for the remainder of the paper}. Consider the framed curves $a_i$ in Figure~\ref{fig_surface}. These give rise to framed curves lying in $\hbdy$, which we also denote by $a_i$.

\begin{theorem} (Theorem 4.7(b) in \cite{thetaTQFT}) \label{thm_thetaskein}
There is a $\mathbb{C}$-linear isomorphism
\begin{equation} \label{eqn_thetaskein}
\thetafn(\Sigma_g)  \cong  \rlskein{\hbdy}, \quad
\theta_{\mu}^{\Pi}  \mapsto  a_1^{\mu_1}\cdots a_g^{\mu_g}.
\end{equation}
This isomorphism is equivariant in the sense that the following diagram commutes:
\begin{displaymath}
\xymatrix{ \End(\thetafn(\Sigma_g))\otimes\thetafn(\Sigma_g) \ar[r] \ar[d]^{\cong} & \thetafn(\Sigma_g) \ar[d]^{\cong} \\ \rlskein{\surfcyl}\otimes\rlskein{\hbdy} \ar[r] & \rlskein{\hbdy} }
\end{displaymath}
Here the top horizontal arrow denotes the usual action of $\End(\thetafn(\Sigma_g))$, whilst the bottom arrow denotes the action arising from the identification \eqref{eqn_collarid}. The right vertical arrow denotes the isomorphism \eqref{eqn_thetaskein}, whilst the left vertical arrow denotes the tensor product of the inverse of \eqref{eqn_opskein} with \eqref{eqn_thetaskein}.
\end{theorem}

Having given a description of the space of classical theta functions that is independent of the complex structure, we now turn to describing the action of the mapping class group in this context. Recall that any element $h\in\mcg{\surf}$ may be represented as surgery on a framed link in $\surfcyl$. Take a framed link $L$ in $\surfcyl$ such that the 3-manifold $M_h$ that is obtained from $\surfcyl$ by surgery along $L$ is homeomorphic to $\surfcyl$ by a homeomorphism (where $M_h$ is the domain and $\surfcyl$ is the codomain) that is the identity on $\surf\times\{1\}$ and $h\in\mcg{\surf}$ on $\surf\times\{0\}$. For those links $L$ having this property, the homeomorphism $h_L:=h$ is well-defined up to isotopy.

To describe the skein in $\rlskein{\surfcyl}$ that is associated to the action \eqref{eqn_hermitejacobiact} of the mapping class group by Theorem~\ref{thm_opskein}, we introduce a special skein.

\begin{definition}
Consider the skein in the solid torus that is depicted in Figure \ref{fig_omega} multiplied by $N^{-\frac{1}{2}}$.
\begin{figure}[ht]
\centering
\scalebox{.30}{\includegraphics{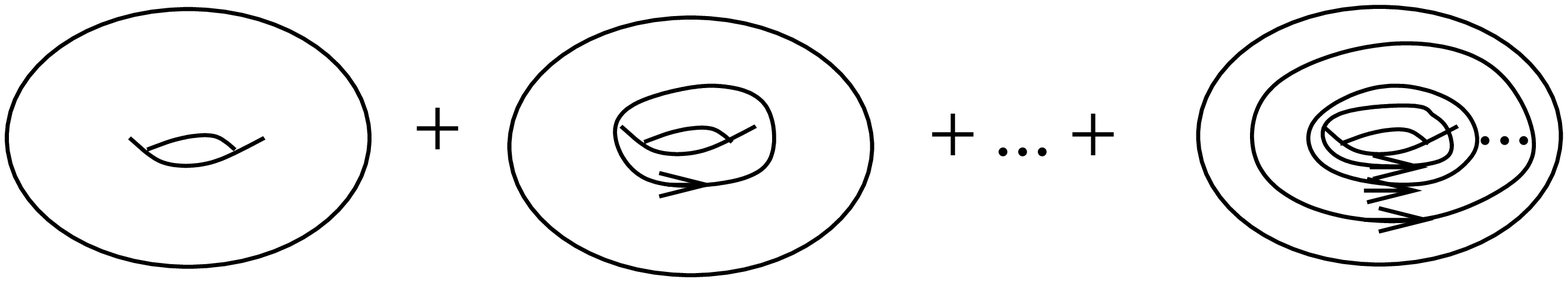}}
\caption{The skein $\Omega$. The sum has $N$ terms.}
\label{fig_omega}
\end{figure}
Given a framed link $L$ in a 3-manifold $M$, let $\Omega(L)\in\rlskein{M}$ denote the skein that is obtained from $L$ by replacing each component of $L$ with the skein from Figure~\ref{fig_omega} using the framing of that component. Specifically, if $L$ is the disjoint union $L=L_1\cup\cdots\cup L_m$ of closed framed curves $L_i$, then
\[ \Omega(L):=N^{-\frac{m}{2}}\sum_{i_1,\ldots,i_m=0}^{N-1} L_1^{\| i_1}\cup L_2^{\| i_2}\cup\cdots\cup L_m^{\| i_m}. \]
Note that $\Omega(L)$ does not depend upon the orientation of the link components of $L$.
\end{definition}

\begin{theorem} (Theorem 8.1 in \cite{thetaTQFT}) \label{thm_mcgskein}
Let $h_L\in\mcg{\surf}$ be a homeomorphism that is represented by surgery on a framed link $L$. The skein in $\rlskein{\surfcyl}/U(1)$ associated to the endomorphism $\rho(h_L)\in U(\thetafn(\Sigma_g))/U(1)$ by Theorem \ref{thm_opskein} is $\Omega(L)$.
\end{theorem}

Consequently, by Theorem \ref{thm_thetaskein}, the projective representation \eqref{eqn_hermitejacobiact} of the mapping class group is modeled on $\rlskein{\hbdy}$ as multiplication by the skein $\Omega(L)\in\rlskein{\surfcyl}$. Having now defined a completely topological model for classical theta functions and the action of the mapping class group, their study may, for the purposes of this paper, be entirely divorced from any dependence upon the complex structure on the surface $\surf$ which we made use of when applying the method of geometric quantization.

\subsection{Surgery diagrams for parameterized 3-cobordisms} \label{sec_surgerycobo}

The constructions of this section are standard (see \cite{bhmv}, \cite{turaev}).

Every compact, connected, oriented 3-manifold without boundary can be represented by surgery on a framed link in $S^3$, and any two surgery diagrams of the same manifold can be transformed into one another by a finite sequence of isotopies and the Kirby moves (k1) and (k2), where (k1) consists of adding/deleting a trivial link component with framing $\pm 1$ and (k2) consists of sliding one link component over the other.

A compact, connected, oriented 3-manifold with boundary can be represented by surgery in the complement of the regular neighborhood of a ribbon graph whose connected components are as shown in Figure~\ref{fig_graphandlebody}. Recall the following definition.

\begin{figure}[h]
\resizebox{.40\textwidth}{!}{\includegraphics{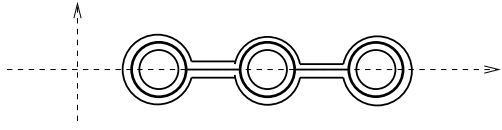}}
\caption{Ribbon graph whose regular neighborhood is a handlebody}
\label{fig_graphandlebody}
\end{figure}

\begin{definition}
A parameterized 3-cobordism
\[ M:\partial_- M \Rightarrow \partial_+ M \]
is a triple $(M,\partial_- M,\partial_+ M)$ where $M$ is an oriented compact 3-manifold and $\partial_- M$, $\partial_+ M$ are oriented closed subsurfaces which partition $\partial M$ such that the orientation of $\partial_+ M$ coincides with that of $\partial M$ while that of $\partial_- M$ does not. The subsurfaces $\partial_{\pm} M$ are each parameterized by fixed orientation preserving homeomorphisms
\begin{equation} \label{eqn_param}
f_-: \bigsqcup_{i=1}^{k_-} \surf[g^-_i] \to \partial_- M \qquad\text{and}\qquad f_+: \bigsqcup_{i=1}^{k_+} \surf[g^+_i] \to \partial_+ M.
\end{equation}
\end{definition}

Two parameterized 3-cobordisms $M:\partial_- M \Rightarrow \partial_+ M$ and $M':\partial_- M' \Rightarrow \partial_+ M'$ can be composed if and only if $\partial_+M$ and $\partial_-M'$ are parameterized by exactly the same surface. In this case, the composition
\begin{eqnarray} \label{eqn_composition}
(M'':\partial_-M'' \Rightarrow \partial_+M''):=(M':\partial_-M' \Rightarrow \partial_+M')\circ(M:\partial_-M \Rightarrow \partial_+M)
\end{eqnarray}
is obtained by gluing $M$ to $M'$ using the parameterizations \eqref{eqn_param}.

For any connected parameterized 3-cobordism $M:\partial_-M \Rightarrow \partial_+M$ there is a surgery diagram for $M$ such that:
\begin{enumerate}
\item
The ribbon graph $\Gamma_{\pm}$ corresponding to  $\partial_\pm M$ lies in the $yz$-plane of ${\mathbb R}^3\subset S^3$ in such a way that its ``circles'' are of the form $(y-j)^2+(z-k)^2=1/9$ and its ``edges'' lie on the line $z=0$ for $\Gamma_-$ and $z=1$ for $\Gamma_+$.
\item
The surgery link $L$ lies entirely inside the slice ${\mathbb R}^2\times(0,1)$.
\item
\begin{enumerate}
\item
The parameterization of $\partial_- M$ is obtained by translating each standard surface $\surf[g_i]$ to the boundary of the regular neighborhood of the $i$th component of $\Gamma_-$. In this way, the curves $a_j$ from Figure~\ref{fig_surface} run \emph{counterclockwise}.
\item
The parameterization of $\partial_+ M$ is obtained by again performing such a translation to the components of $\Gamma_+$, followed by a reflection in the plane $z=1$. In this way, the curves $a_j$ from Figure~\ref{fig_surface} run \emph{clockwise}.
\end{enumerate}
\end{enumerate}

\begin{remark} \label{rem_kirby}
Any two surgery diagrams that represent the same parameterized 3-cobordism can be transformed into one another by a sequence of isotopies, of Kirby moves (k1) and (k2) performed
on the surgery link, and by slides of the edges of the graph along the surgery link components.
\end{remark}

As an example, a surgery diagram for the cylinder
\[ \surfcyl:\surf\times\{0\}\Rightarrow\surf\times\{1\} \]
whose ends are parameterized by the identity parameterizations is shown in Figure \ref{fig_surgcyl}.

\begin{figure}[h]
\centerline{\psfig{file=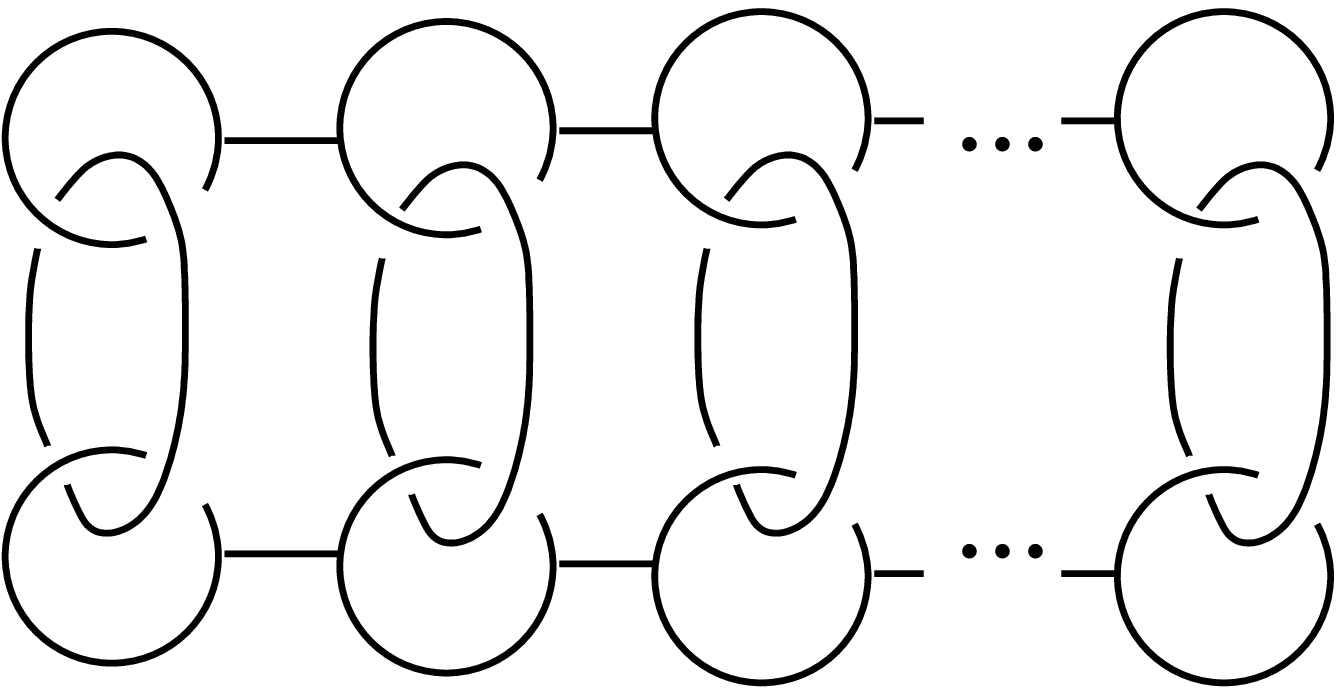,width=1.8in}}
\caption{Surgery diagram for $\surfcyl$, which consists of $g$ surgery curves.}
\label{fig_surgcyl}
\end{figure}

In what follows, we will require the definition of a nondegenerate symmetric bilinear form $\innprod$ on $\rlskein{\hbdy}$ which differs from the Hermitian form defined by \eqref{eqn_thetaform}. Given skeins $L$, $L'\in\rlskein{\hbdy}$, embed the skein $L$ in a regular neighborhood of the bottom ribbon graph of Figure \ref{fig_surgcyl} by a translation. Similarly, using just such a translation, embed the skein $L'$ in a regular neighborhood of the top ribbon graph. Replace the surgery curves of Figure \ref{fig_surgcyl} by the skein $\Omega$. The result is a skein
\begin{equation} \label{eqn_skeinform}
\langle L,L' \rangle\in\rlskein{S^3}=\mathbb{C}.
\end{equation}
One may check by direct calculation, using the standard formula for the sum of powers of a primitive $N$th root of unity, that for $(i_1,\ldots,i_g)$, $(j_1,\ldots,j_g)\in\mathbb{Z}_N^g$;
\begin{equation} \label{eqn_dualbasis}
\langle a_1^{i_1}\cdots a_g^{i_g},a_1^{j_1}\cdots a_g^{j_g} \rangle = \left\{\begin{array}{ll} N^{\frac{g}{2}}, & (i_1,\ldots,i_g)=(j_1,\ldots,j_g) \\ 0, & (i_1,\ldots,i_g)\neq(j_1,\ldots,j_g) \end{array}\right\}.
\end{equation}
Hence \eqref{eqn_skeinform} is nondegenerate.

The following result is due essentially to Turaev \cite{reshetikhinturaev}, \cite{turaev}.

\begin{theorem} \label{thm_turaevglue}
If $(L,\Gamma_-,\Gamma_+)$ and $(L',\Gamma_-',\Gamma_+')$ are surgery diagrams for the connected parameterized 3-cobordisms
\[ M:\partial_-M\Rightarrow\partial _+M \quad\text{and}\quad M':\partial_-M'\Rightarrow\partial _+M' \]
respectively, then a surgery diagram for the composition \eqref{eqn_composition} is obtained as follows:
\begin{enumerate}
\item
place the surgery diagram $(L,\Gamma_-,\Gamma_+)$ in ${\mathbb R}^3\subset S^3$ such that the ``circles'' of $\Gamma_-$ are of the form $(y-j)^2+(z-0)^2=1/9$ and its ``edges'' lie on $z=0$, and the ``circles'' of $\Gamma_+$ are of the form $(y-j)^2+(z-1)^2=1/9$ and its ``edges'' lie on $z=1$;
\item
place the surgery diagram $(L',\Gamma_-',\Gamma_+')$ in ${\mathbb R}^3\subset S^3$ such that the ``circles'' of $\Gamma_-'$ are of the form $(y-j)^2+(z-1)^2=1/9$ and its ``edges'' lie on $z=1$, and the ``circles'' of $\Gamma_+'$ are of the form $(y-j)^2+(z-2)^2=1/9$ and its ``edges'' lie on $z=2$; so the ``circles'' and ``edges'' of $\Gamma_+$ and $\Gamma_-'$ overlap;
\item
place a horizontal surgery circle in the plane $z=1$ around all but one of the components of $\Gamma_+=\Gamma_-'$ and endow it with the framing that is parallel to this plane;
\item
delete the edges of $\Gamma_+=\Gamma_-'$ and interpret the ``circles'' of these overlapping graphs as surgery curves.
\end{enumerate}
\end{theorem}
\

If $(L,\Gamma_-,\Gamma_+)$ and $(L',\Gamma_-',\Gamma_+')$ are the surgery diagrams of the two cobordisms, then Theorem \ref{thm_turaevglue} tells us how to associate a surgery diagram $(L'',\Gamma_-,\Gamma_+')$ to the composition. We denote the link  $L''$ by $L'\circ L$.

\begin{remark} \label{rem_turaevglue}
Theorem \ref{thm_turaevglue} extends to 4-dimensional handlebodies, cf. Theorem IV.4.3 of \cite{turaev}. If two 4-manifolds $W$ and $W'$ are obtained by gluing 2-handles to the 4-ball along framed links $L$ and $L'$ respectively, then the 4-manifold $W''$ obtained from $W$ and $W'$ by gluing a regular neighborhood of the top ribbon graph $\Gamma_+$ in $\partial W$ to a regular neighborhood of the bottom ribbon graph $\Gamma_-'$ in $\partial W'$ along a reflection is the result of gluing 2-handles to the 4-ball along the link $L'':=L'\circ L$.
\end{remark}

An illustration of the composition is shown in Figure~\ref{fig_gluingcobsurg}. Note that in this picture we can slide the horizontal circle at the very right over the horizontal circle in the middle to obtain the horizontal circle that links the first two vertical circles. We can also flip the horizontal circle in the middle so that it links the two surgery circles on the left. Thus it is irrelevant on which two of the three possible locations we place the horizontal circles.

\begin{figure}[h]
\resizebox{.60\textwidth}{!}{\includegraphics{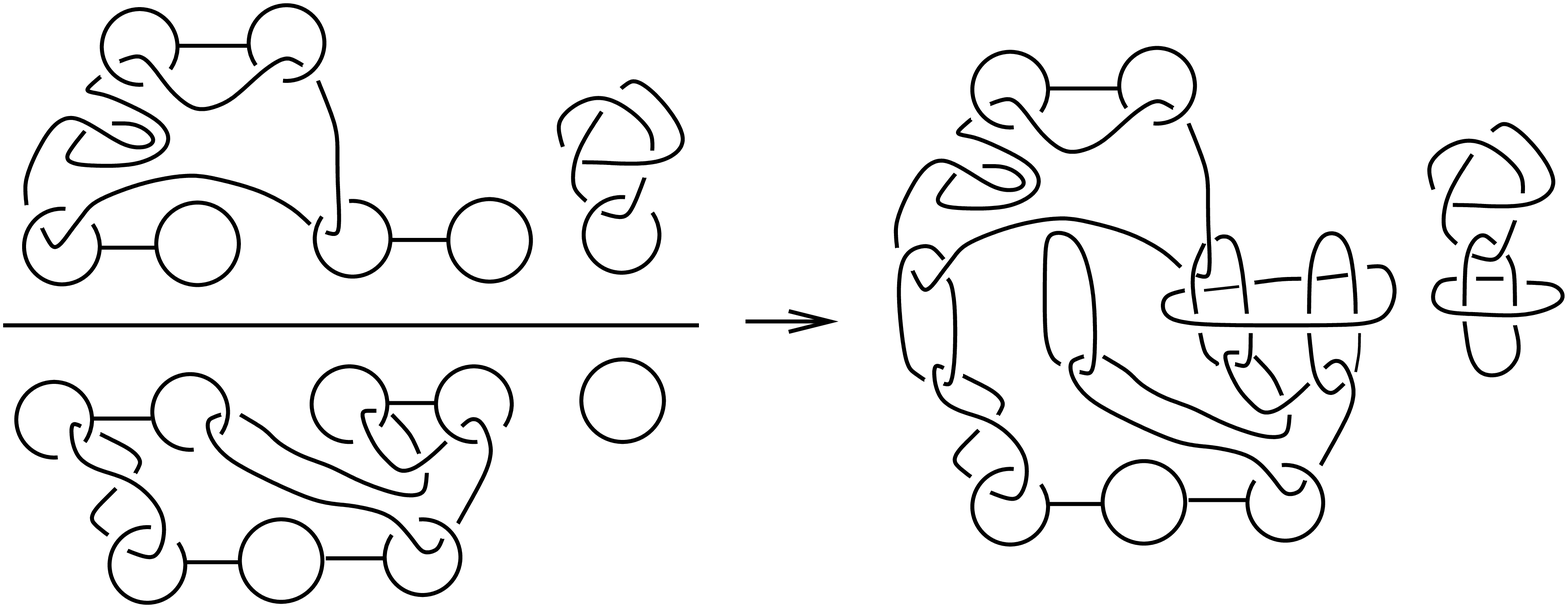}}
\caption{Gluing of surgery diagrams for parameterized 3-cobordisms.}
\label{fig_gluingcobsurg}
\end{figure}

\section{The resolution of the projective ambiguity of the representation of the mapping class group} \label{sec_resolve}

The algebraic construction that transforms  the projective representation of the mapping class group into a true representation is well-known; it is the Segal-Shale-Weil representation. Our purpose here is to give a topological solution to the projectivity issue. The construction is inspired by \cite{walker} and the tools are from \cite{bhmv} and \cite{turaev}. First, we must introduce  the Maslov index (see \cite{kashiwarashapira}).

\begin{definition}
Let $\mathbf{L}_1$, $\mathbf{L}_2$ and $\mathbf{L}_3$ be three Lagrangian subspaces of a real symplectic vector space $(V,\omega)$. The \emph{Maslov index} $\tau(\mathbf{L}_1,\mathbf{L}_2,\mathbf{L}_3)$ is the signature of the symmetric bilinear form
\begin{displaymath}
([\mathbf{L}_1+\mathbf{L}_2]\cap\mathbf{L}_3)^{\otimes 2} \to  \mathbb{R}, \quad
(x_1+x_2)\otimes x_3 \mapsto  \omega(x_2,x_3);\mbox{ where }x_i\in {\mathbf{L}_i}.
\end{displaymath}
\end{definition}

Now, we may make the following definition.

\begin{definition} \label{def_extmcg}
The \emph{extended mapping class group} of the surface $\surf$ is the ${\mathbb Z}$-extension of $\mcg{\surf}$ defined by the multiplication rule
\begin{eqnarray*}
(h',n')(h,n)=(h' h, n+n'+\tau(h'_*h_*(\mathbf{L}_g),h'_*(\mathbf{L}_g),\mathbf{L}_g)),
\end{eqnarray*}
where
\begin{equation} \label{eqn_stdlag}
\mathbf{L}_g:=\ker[H_1(\surf,\mathbb{R}) \to H_1(\hbdy,\mathbb{R})]
\end{equation}
is the Lagrangian subspace spanned by the curves $b_i$ from Figure \ref{fig_surface}. We denote the extended mapping class group by $\emcg{\surf}$.
\end{definition}

Given a framed link $F$ in $S^3$, denote its signature by $\sigma(F)$. Recall that this is the signature of the 4-manifold obtained by attaching 2-handles to the 4-ball along regular neighborhoods of the components of $F$. If $L$ is a framed link in the cylinder $\surfcyl$, we may embed this link in $S^3$ using the surgery presentation for this cylinder given in Figure \ref{fig_surgcyl}. Deleting the ribbon graphs at each end of this surgery presentation yields a link $L_*$ whose components are formed from the components of $L$ and the $g$ annuli of Figure \ref{fig_surgcyl}. An example is shown in Figure \ref{fig_surgmapdiag}. Define $\sigma_*(L)$ to be the signature of this link $L_*$.

\begin{figure}[h]
\centerline{\psfig{file=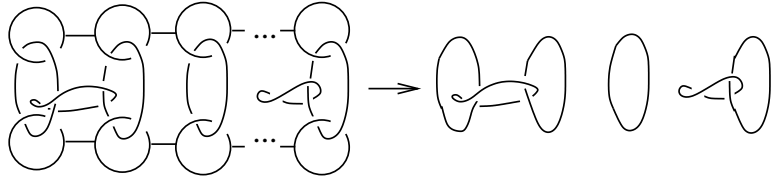,width=4in}}
\caption{Deleting the ribbon graphs yields a link $L_*$ formed from $L$ and the $g$ annuli.}
\label{fig_surgmapdiag}
\end{figure}

\begin{theorem} \label{thm_masgrpact}
There is a well-defined representation of the extended mapping class group $\emcg{\surf}$ on the space of theta functions $\rlskein{\hbdy}$ given by
\begin{equation} \label{eqn_masgrpact}
\begin{array}{ccc}
\emcg{\surf} & \to & \rlskein{\surfcyl}, \\
(h,n) & \mapsto & \mathcal{F}(h,n):=e^{-\frac{\pi i}{4}(n+\sigma_*(L))}\Omega(L);
\end{array}
\end{equation}
where $L$ is any surgery presentation of $h$.
\end{theorem}

\begin{remark}
Here we have used $\rlskein{\surfcyl}\cong\End(\thetafn(\surf))$, cf. Theorem \ref{thm_opskein}.
\end{remark}

\begin{proof}
To prove that $\mathcal{F}(h,n)$ does not depend on the choice of surgery presentation $L$ for $h$; it suffices, since \eqref{eqn_skeinform} is nondegenerate, to prove that the expression
\begin{equation} \label{eqn_skinvariance}
e^{-\frac{\pi i}{4}\sigma_*(L)} \langle \Omega(L)\cdot a,a'\rangle; \quad a,a'\in\rlskein{\hbdy}
\end{equation}
is independent of $L$. But \eqref{eqn_skinvariance} is just the skein in $S^3$ obtained from $e^{-\frac{\pi i}{4}\sigma(L_*)}\Omega(L_*)$ by inserting $a$ at the bottom of Figure \ref{fig_surgcyl} and $a'$ at the top. Standard Kirby calculus arguments (cf. Remark \ref{rem_kirby}) and some simple calculations involving Gauss sums now imply the claim.

Now, we will show that \eqref{eqn_masgrpact} is multiplicative. Suppose that $h'$ and $h$ are two homeomorphisms represented by surgery on framed links $L'$ and $L$ respectively. Then the composite $h'':=h'\circ h$ is represented by surgery on the framed link $L'':=L'\cdot L$ in which $L'$ is placed on top of $L$. We obviously have the identity
\[ \Omega(L'') = \Omega(L')\cdot\Omega(L), \]
so we only need to calculate the signatures of the corresponding links.

Consider the link $L'_*\circ L_*$ defined by Theorem \ref{thm_turaevglue} and shown on the left of Figure \ref{fig_linkisotopy} together with two accompanying ribbon graphs. By Remark \ref{rem_turaevglue} and Wall's formula for the nonadditivity of the signature \cite{wall}, the signature of this link is
\[ \sigma(L'_*\circ L_*) = \sigma(L'_*)+\sigma(L_*) - \tau(h'_*h_*(\mathbf{L}_g),h'_*(\mathbf{L}_g),\mathbf{L}_g). \]
Performing surgery on the $3g$ annuli of Figure \ref{fig_linkisotopy} in the complement of a neighborhood of the ribbon graphs yields a cylinder $\surfcyl$, so we may slide the link $L'$ on top of $L$ yielding the link shown on the right of Figure \ref{fig_linkisotopy}. Note that this does not change the signature of course, as it just amounts to an isotopy of the placement of the 2-handles on the 4-manifold. Another application of Wall's formula yields that the signature of the link on the right, obtained after deleting the two ribbon graphs, is just
\[ \sigma(\emptyset_*\circ L''_*)=\sigma(\emptyset_*) + \sigma(L''_*) -\tau(h''_*(\mathbf{L}_g),\mathbf{L}_g,\mathbf{L}_g) = \sigma(L''_*). \]
Hence
\[ \sigma_*(L'') = \sigma(L'_*\circ L_*) = \sigma_*(L')+\sigma_*(L) - \tau(h'_*h_*(\mathbf{L}_g),h'_*(\mathbf{L}_g),\mathbf{L}_g), \]
which implies the result.

\begin{figure}[h]
\resizebox{.60\textwidth}{!}{\includegraphics{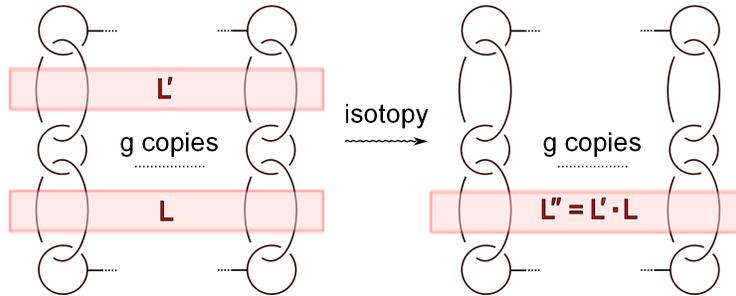}}
\caption{After performing surgery on the $3g$ annuli, we may slide $L'$ on top of $L$.}
\label{fig_linkisotopy}
\end{figure}
\end{proof}

\section{The Topological Quantum Field Theory associated to classical theta functions} \label{sec_tqft}

\subsection{The category of extended surfaces and the functor of states}

In this section we define an extended category of surfaces which forms part of the background for our TQFT and the corresponding functor of states.

\subsubsection{The category of extended surfaces}

Intuitively, the category of extended surfaces is a larger universe that contains all extended mapping class groups. In this way, the projective anomaly of the representation \eqref{eqn_hermitejacobiact} of the mapping class group will be incorporated into the definition of a TQFT.

\begin{definition} \label{def_extsurf}
The category of extended surfaces has
\begin{enumerate}
\item
as objects the extended surfaces, which consist of pairs $\mathbf{\Sigma}:=(\Sigma,\mathbf{L}_{\Sigma})$, where $\Sigma$ is a closed oriented surface and $\mathbf{L}_{\Sigma}$ is a Lagrangian subgroup of $H_1(\Sigma,\mathbb{Z})$ that splits as a direct sum of Lagrangian subgroups amongst the connected components of $\Sigma$. Recall that a Lagrangian subgroup of $H_1(\Sigma,\mathbb{Z})$ is a maximal subgroup amongst those on which the intersection form vanishes. If $\mathbf{L}$ is a Lagrangian subgroup, then $\mathbf{L}_{\mathbb{R}}:=\mathbf{L}\otimes_{\mathbb{Z}}\mathbf{R}$ is a Lagrangian subspace of $H_1(\Sigma,\mathbb{R})$.
\item
as morphisms the extended homeomorphisms. An extended homeomorphism $\mathbf{h}:\mathbf{\Sigma}\to\mathbf{\Sigma}'$ is a pair $\mathbf{h}:=(h,n_h)$, where $h:\Sigma\to\Sigma'$ is an orientation preserving homeomorphism and $n_h$ is an integer, which we refer to as the \emph{weight} of $\mathbf{h}$. The homeomorphism $h$ need not preserve the Lagrangian subspaces. Extended homeomorphisms
\[ (h,n_h): (\Sigma,\mathbf{L}) \to (\Sigma',\mathbf{L}') \quad\text{and}\quad (h',n'_h): (\Sigma',\mathbf{L}') \to (\Sigma'',\mathbf{L}'') \]
may be composed according to the rule
\begin{equation} \label{eqn_mascompose}
(h',n'_h) \circ (h,n_h) = (h'\circ h,n'_h+n_h+\tau(h'_*h_*(\mathbf{L}_{\mathbb{R}}),h'_*(\mathbf{L}'_{\mathbb{R}}),\mathbf{L}''_{\mathbb{R}})).
\end{equation}
\end{enumerate}
We denote the standard extended surface of genus $g$ by $\esurf:=(\surf,\mathbf{L}_g)$, where $\mathbf{L}_g$ is as defined by \eqref{eqn_stdlag}. We denote the empty extended surface by $\mathbf{\Phi}:=(\emptyset,\{0\})$.
\end{definition}

\begin{remark}
The usual cocycle identity (cf. Proposition 1.5.8 of \cite{lionvergne}) for the Maslov index $\tau$ implies that \eqref{eqn_mascompose} is associative. All morphisms in this category are invertible, with $(h,n)^{-1}=(h^{-1},-n)$. Note that the extended mapping class group $\emcg{\surf}$ defined by Definition \ref{def_extmcg} may be identified as the quotient under the isotopy relation of the group $\Aut(\esurf)$ of automorphisms in this category of the standard extended surface of genus $g$.
\end{remark}

\begin{remark} \label{rem_param}
In principle, it is possible to work in the category in which we permit our surfaces to be decorated by any Lagrangian subspace of $H_1(\Sigma,\mathbb{R})$, but doing so provides no extra utility. In fact, restricting the choice of subspace as above will significantly simplify Maslov index calculations, as given any connected extended surface $(\Sigma,\mathbf{L})$, we may always find a parameterization $f:\surf\to\Sigma$ such that $f_*(\mathbf{L}_g)=\mathbf{L}$. In what follows, we will omit the subscript $\mathbb{R}$ from the Maslov index in \eqref{eqn_mascompose}, as no confusion can possibly arise.
\end{remark}

Disjoint unions of extended surfaces may be formed in the obvious way, with the weight of a disjoint union of morphisms being defined as the sum of the individual weights.

\subsubsection{The functor of states}

Having defined our category of surfaces, we may proceed to define the functor of states. We employ the following standard idea, which is entirely analogous to covering a manifold with a multitude of coordinate charts. Suppose that $\mathbf{\Sigma}=(\Sigma,\mathbf{L}_{\Sigma})$ is an extended surface. A \emph{parameterization} of $\mathbf{\Sigma}$ is an extended homeomorphism
\begin{equation} \label{eqn_extparam}
\mathbf{f}:\bigsqcup_{i=1}^k\esurf[g_i]\to\mathbf{\Sigma}.
\end{equation}
Consider the small category $\mathscr{P}(\mathbf{\Sigma})$ whose objects consist of parameterizations \eqref{eqn_extparam} of $\mathbf{\Sigma}$, and whose set of morphisms between two parameterizations $\mathbf{f}$ and $\mathbf{f}'$ consists of the one and only one extended homeomorphism between the domains of $\mathbf{f}$ and $\mathbf{f}'$ that commutes with the parameterizations. Consider the functor $\mathscr{L}$ from $\mathscr{P}(\mathbf{\Sigma})$ to complex vector spaces that assigns to a parameterization \eqref{eqn_extparam} the vector space
\begin{equation} \label{eqn_paramfun}
\mathscr{L}(\mathbf{f}):=\bigotimes_{i=1}^k \rlskein{\hbdy[g_i]},
\end{equation}
and assigns to a morphism of parameterizations the representation \eqref{eqn_masgrpact}, which is extended to disjoint unions of surfaces in the obvious way, permuting the factors of \eqref{eqn_paramfun} as necessary.

\begin{definition} \label{def_states}
We define a functor $\mathscr{V}$ from extended surfaces to complex vector spaces by the following limit,
\[ \mathscr{V}(\mathbf{\Sigma}) := \dilim{\mathbf{f}\in\mathscr{P}(\mathbf{\Sigma})}{\mathscr{L}(\mathbf{f})}. \]
Consequently, we have a commuting system of maps;
\[ \iota_{\mathbf{f}}:\mathscr{L}(\mathbf{f})\to\mathscr{V}(\mathbf{\Sigma}), \quad\mathbf{f}\in\mathscr{P}(\mathbf{\Sigma}). \]
Since this limit is taken over the groupoid $\mathscr{P}(\mathbf{\Sigma})$, the maps $\iota_{\mathbf{f}}$ are in fact all \emph{isomorphisms}, and the notions of direct limit and inverse limit coincide.

Given an extended homeomorphism $\mathbf{h}:\mathbf{\Sigma}\to\mathbf{\Sigma}'$, the basic property of a limit asserts the existence of a unique map $\mathscr{V}(\mathbf{h})$ making the following diagram commute for all parameterizations $\mathbf{f}\in\mathscr{P}(\mathbf{\Sigma})$:
\[ \xymatrix{ \mathscr{V}(\mathbf{\Sigma}) \ar[r]^{\mathscr{V}(\mathbf{h})} & \mathscr{V}(\mathbf{\Sigma}') \\ \mathscr{L}(\mathbf{f}) \ar[u]^{\iota_{\mathbf{f}}} \ar@{=}[r] & \mathscr{L}(\mathbf{h}\circ\mathbf{f}) \ar[u]^{\iota_{\mathbf{h}\circ\mathbf{f}}} } \]
\end{definition}

In the following proposition, we mostly adopt the point of view of Turaev \cite[\S III.1.2]{turaev}.

\begin{proposition} \label{prop_funcstates}
The functor $\mathscr{V}$ from the category of extended surfaces to complex vector spaces satisfies the following axioms:
\begin{enumerate}
\item
For any two extended surfaces $\mathbf{\Sigma}$ and $\mathbf{\Sigma}'$, there is a canonical identification isomorphism
\[ \mathscr{V}(\mathbf{\Sigma}\sqcup\mathbf{\Sigma}') = \mathscr{V}(\mathbf{\Sigma})\otimes\mathscr{V}(\mathbf{\Sigma}') \]
such that:
\begin{enumerate}
\item
This identification is natural in $\mathbf{\Sigma}$ and $\mathbf{\Sigma}'$.
\item \label{axiom_permute}
The diagram
\[ \xymatrix{ \mathscr{V}(\mathbf{\Sigma}\sqcup\mathbf{\Sigma}') \ar@{=}[r] \ar[d] & \mathscr{V}(\mathbf{\Sigma})\otimes\mathscr{V}(\mathbf{\Sigma}') \ar[d]^P \\ \mathscr{V}(\mathbf{\Sigma}'\sqcup\mathbf{\Sigma}) \ar@{=}[r] & \mathscr{V}(\mathbf{\Sigma}')\otimes\mathscr{V}(\mathbf{\Sigma}) } \]
commutes, where $P(x\otimes y):=y\otimes x$ and the vertical map on the left is induced by the identification $\mathbf{\Sigma}\sqcup \mathbf{\Sigma}' = \mathbf{\Sigma}'\sqcup\mathbf{\Sigma}$.
\item
The diagram
\[ \xymatrix{ \mathscr{V}(\mathbf{\Sigma})\otimes\mathscr{V}(\mathbf{\Sigma}')\otimes\mathscr{V}(\mathbf{\Sigma}'') \ar@{=}[r] \ar@{=}[d] & \mathscr{V}(\mathbf{\Sigma}\sqcup\mathbf{\Sigma}')\otimes\mathscr{V}(\mathbf{\Sigma}'') \ar@{=}[d] \\ \mathscr{V}(\mathbf{\Sigma})\otimes\mathscr{V}(\mathbf{\Sigma}'\sqcup\mathbf{\Sigma}'') \ar@{=}[r] & \mathscr{V}(\mathbf{\Sigma}\sqcup\mathbf{\Sigma}'\sqcup\mathbf{\Sigma}'') } \]
commutes.
\end{enumerate}
\item
The functor $\mathscr{V}$ is unital in that $\mathscr{V}(\mathbf{\Phi})=\mathbb{C}$ and the diagram
\[ \xymatrix{ \mathscr{V}(\mathbf{\Phi}\sqcup\mathbf{\Sigma}) \ar@{=}[r] \ar@{=}[d] & \mathscr{V}(\mathbf{\Phi})\otimes\mathscr{V}(\mathbf{\Sigma}) \ar@{=}[d] \\ \mathscr{V}(\mathbf{\Sigma}) \ar@{=}[r] & \mathbb{C}\otimes\mathscr{V}(\mathbf{\Sigma}) } \]
is commutative for any extended surface $\mathbf{\Sigma}$.
\item
There is a nonzero complex number $\anomaly$, called the \emph{anomaly}, such that for any extended surface $\mathbf{\Sigma}$,
\[ \mathscr{V}(\id_{\Sigma},1) = \anomaly\cdot\id_{\mathscr{V}(\mathbf{\Sigma})}. \]
For the functor $\mathscr{V}$, this anomaly is given by $a:=e^{-\frac{\pi i}{4}}$.
\end{enumerate}
\end{proposition}

\begin{proof}
Given extended surfaces $\mathbf{\Sigma}$ and $\mathbf{\Sigma}'$ we have
\begin{displaymath}
\begin{split}
\mathscr{V}(\mathbf{\Sigma}\sqcup\mathbf{\Sigma}') &= \dilim{\mathbf{g}\in\mathscr{P}(\mathbf{\Sigma}\sqcup\mathbf{\Sigma}')}{\mathscr{L}(\mathbf{g})} = \dilim{(\mathbf{f},\mathbf{f}')\in\mathscr{P}(\mathbf{\Sigma})\times\mathscr{P}(\mathbf{\Sigma}')}{\mathscr{L}(\mathbf{f}\sqcup\mathbf{f}')}, \\
&= \dilim{(\mathbf{f},\mathbf{f}')\in\mathscr{P}(\mathbf{\Sigma})\times\mathscr{P}(\mathbf{\Sigma}')}{\left(\mathscr{L}(\mathbf{f})\otimes\mathscr{L}(\mathbf{f}')\right)} = \left(\dilim{\mathbf{f}\in\mathscr{P}(\mathbf{\Sigma})}{\mathscr{L}(\mathbf{f})}\right)\otimes\left(\dilim{\mathbf{f}'\in\mathscr{P}(\mathbf{\Sigma}')}{\mathscr{L}(\mathbf{f}')}\right), \\
&= \mathscr{V}(\mathbf{\Sigma})\otimes\mathscr{V}(\mathbf{\Sigma}').
\end{split}
\end{displaymath}
The remaining axioms are verified in a similarly tautological fashion.
\end{proof}

\subsection{The category of framed 3-cobordisms and the topological quantum field theory}

The topological quantum field theory is a global object which incorporates theta functions, the Schr\"odinger representation of the finite Heisenberg group, the representation \eqref{eqn_hermitejacobiact}, and brings together surfaces of all genera.

\subsubsection{The category of framed 3-cobordisms} \label{sec_framedbord}

We introduce a finer category than the category of parameterized 3-cobordisms considered in Section \ref{sec_surgerycobo}. Firstly, because of our solution \eqref{eqn_masgrpact} to the issue of the projective anomaly in the representation \eqref{eqn_hermitejacobiact} of the mapping class group, the boundaries of our cobordisms should be extended surfaces and the cobordisms themselves should be framed by integers. Secondly, because theta functions and the Schr\"odinger representation are modeled by the skein theory of framed links, these cobordisms should contain oriented framed links in their interior.

\begin{definition}
A framed 3-cobordism
\[ \mathbf{M}:\partial_-\mathbf{M}\Rightarrow\partial_+\mathbf{M} \]
is a tuple $\mathbf{M}:=(M,\partial_-\mathbf{M},\partial_+\mathbf{M},L_M,n_M)$, where:
\begin{itemize}
\item
$M$ is a compact oriented 3-dimensional manifold;
\item
$\partial_-\mathbf{M}=(\partial_-M,\mathbf{L}_-)$ and $\partial_+ \mathbf{M}=(\partial_+ M,\mathbf{L}_+)$ are extended surfaces such that $\partial_-M$ and $\partial_+M$ are subsurfaces of $\partial M$, partitioning it such that $\partial_+M$ has the same orientation as $\partial M$ while $\partial_- M$ has the opposite one;
\item
$L_M$ is an oriented framed link embedded in the interior of $M$;
\item
$n_M$ is an integer, called the \emph{weight} of $\mathbf{M}$.
\end{itemize}

A homeomorphism of framed 3-cobordisms is an orientation preserving homeomorphism of the underlying 3-manifolds that preserves the structures described above. According to Definition \ref{def_extsurf}, this map need \emph{not} preserve the Lagrangian subspaces of course. Disjoint unions of framed 3-cobordisms are formed in the obvious way, by taking the disjoint union of the underlying 3-manifolds, with the weight of the union being defined as the sum of the weights of each individual piece.
\end{definition}

We may glue two framed 3-cobordisms
\[ \mathbf{M}:\partial_-\mathbf{M}\Rightarrow\partial_+\mathbf{M} \quad\text{and}\quad \mathbf{M}':\partial_-\mathbf{M}'\Rightarrow\partial_+\mathbf{M}' \]
along an extended homeomorphism $\mathbf{h}:\partial_+\mathbf{M}\to\partial_-\mathbf{M}'$. The composition
\[ (\mathbf{M}'':\partial_-\mathbf{M}''\Rightarrow\partial_+\mathbf{M}'') := (\mathbf{M}':\partial_-\mathbf{M}'\Rightarrow\partial_+\mathbf{M}') \circ_{\mathbf{h}} (\mathbf{M}:\partial_-\mathbf{M}\Rightarrow\partial_+\mathbf{M}) \]
is formed in the obvious way, by gluing the underlying 3-manifolds together along the homeomorphism $h$ and taking $L_{M''}$ to be the union of the links $L_{M'}$ and $L_M$. The only part that really requires any explanation is how the weight $n_{M''}$ of $\mathbf{M}''$ is defined. Consider the following Lagrangian subspaces:
\begin{itemize}
\item
the subspace $N_M(\mathbf{L}_-)$ of $H_1(\partial_+M,\mathbb{R})$, which consists of all those elements $y$ for which there is an $x\in\mathbf{L}_-$ such that $x$ and $y$ are homologous in $H_1(M,\mathbb{R})$;
\item
the subspace $N^{M'}(\mathbf{L}'_+)$ of $H_1(\partial_-M',\mathbb{R})$, which consists of all those elements $y$ for which there is an $x\in\mathbf{L}'_+$ such that $x$ and $y$ are homologous in $H_1(M',\mathbb{R})$.
\end{itemize}
Then the weight of the composition is defined by
\[ n_{M''}:= n_M + n_{M'} + n_h + \tau\left(h_*(N_M(\mathbf{L}_-)),h_*(\mathbf{L}_+),N^{M'}(\mathbf{L}'_+)\right) +\tau\left(h_*(\mathbf{L}_+),\mathbf{L}'_-,N^{M'}(\mathbf{L}'_+)\right). \]
As is well-known, the cocycle identity for the Maslov index $\tau$ implies that gluing cobordisms is associative. As is also commonly understood, this composition rule is motivated by Wall's formula for the nonadditivity of the signature, as we will see in the proof of Theorem \ref{thm_tqft}.

\subsubsection{The topological quantum field theory}

We now bring the theory of theta functions under the unified framework of topological quantum field theory. The  main result is the \emph{existence} and \emph{uniqueness} of a topological quantum field theory incorporating theta functions. One of the points we wish to emphasize is that this TQFT is essentially fixed by the representation \eqref{eqn_hermitejacobiact} of the mapping class group. This is in turn fixed by the exact Egorov identity \eqref{eqn_egorov} and the principles of quantization.

The construction is in the spirit of \cite{bhmv}. We begin with the definition of a TQFT, following closely the axioms of Turaev \cite{turaev}.

\begin{definition} \label{def_tqft}
A Topological Quantum Field Theory with anomaly $\anomaly\in\mathbb{C}^{\times}$ is a pair $(\mathscr{V},Z)$, where $\mathscr{V}$ is any functor from the category of extended surfaces to complex vector spaces satisfying the axioms of Proposition \ref{prop_funcstates}, and $Z$ is a mapping which assigns to every framed 3-cobordism $\mathbf{M}:\partial_-\mathbf{M}\Rightarrow\partial_+\mathbf{M}$, a $\mathbb{C}$-linear map
\[ Z(\mathbf{M}):\mathscr{V}(\partial_-\mathbf{M})\to\mathscr{V}(\partial_+\mathbf{M}) \]
satisfying the following axioms:
\begin{enumerate}

\item \label{axiom_natural}
If
\[ \mathbf{M}:\partial_-\mathbf{M}\Rightarrow\partial_+\mathbf{M} \quad\text{and}\quad \mathbf{M}':\partial_-\mathbf{M}'\Rightarrow\partial_+\mathbf{M}' \]
are two framed 3-cobordisms and $f:\mathbf{M}\to\mathbf{M}'$ is a homeomorphism between these framed 3-cobordisms (which implies that they have the same weight), then the diagram
\[ \xymatrix{ \mathscr{V}(\partial_-\mathbf{M}) \ar^{Z(\mathbf{M})}[r] \ar[d]_{\mathscr{V}(\mathbf{f}_{|\partial_-\mathbf{M}})} & \mathscr{V}(\partial_+\mathbf{M}) \ar[d]^{\mathscr{V}(\mathbf{f}_{|\partial_+\mathbf{M}})} \\ \mathscr{V}(\partial_-\mathbf{M}') \ar^{Z(\mathbf{M}')}[r] & \mathscr{V}(\partial_+\mathbf{M}') } \]
commutes; where
\begin{displaymath}
\begin{split}
\mathbf{f}_{|\partial_-\mathbf{M}} &:= \left(f_{|\partial_-M},\tau\left(N^{M'}(\mathbf{L}'_+),\mathbf{L}'_-,f_{|\partial_- M}^*(\mathbf{L}_-)\right)\right), \\
\mathbf{f}_{|\partial_+\mathbf{M}} &:= \left(f_{|\partial_+M},\tau\left(f_{|\partial_+ M}^*N_M(\mathbf{L}_-),\mathbf{L}'_+,f_{|\partial_+ M}^*(\mathbf{L}_+)\right)\right).
\end{split}
\end{displaymath}
Here $N_M$ and $N^{M'}$ have the meaning assigned to them above in Section \ref{sec_framedbord}.

\item \label{axiom_cobunion}
If $\mathbf{M}$ and $\mathbf{M}'$ are two framed 3-cobordisms, then the following diagram commutes:
\[ \xymatrix{ \mathscr{V}(\partial_-\mathbf{M}\sqcup\partial_-\mathbf{M}') \ar[rrr]^{Z(\mathbf{M}\sqcup\mathbf{M}')} \ar@{=}[d] &&& \mathscr{V}(\partial_+\mathbf{M}\sqcup\partial_+\mathbf{M}') \ar@{=}[d] \\ \mathscr{V}(\partial_-\mathbf{M})\otimes\mathscr{V}(\partial_-\mathbf{M}') \ar[rrr]^{Z(\mathbf{M})\otimes Z(\mathbf{M}')} &&& \mathscr{V}(\partial_+\mathbf{M})\otimes\mathscr{V}(\partial_+\mathbf{M}') } \]

\item \label{axiom_gluing}
If $\mathbf{M}$ and $\mathbf{M}'$ are two framed 3-cobordisms and $\mathbf{h}:\partial_+\mathbf{M}\to\partial_-\mathbf{M}'$ is an extended homeomorphism, then the map assigned to the cobordism $\mathbf{M}''$ that is obtained by gluing the cobordism $\mathbf{M}$ to the cobordism $\mathbf{M}'$ along the homeomorphism $\mathbf{h}$ is given by
\[ Z(\mathbf{M}'')=Z(\mathbf{M}')\circ\mathscr{V}(\mathbf{h})\circ Z(\mathbf{M}). \]

\item \label{axiom_norm}
If $\mathbf{\Sigma}$ is an extended surface, we may consider the cylinder 3-cobordism
\[ \mathbf{C}[\mathbf{\Sigma},n]:=\left(\Sigma\times [0,1],(\Sigma\times\{0\},\mathbf{L}_{\Sigma}),(\Sigma\times\{1\},\mathbf{L}_{\Sigma}),\emptyset,n\right) \]
of weight $n$, where $\Sigma\times [0,1]$ is given the product orientation. The following condition must be satisfied:
\[ Z(\mathbf{C}[\mathbf{\Sigma},n]) = \anomaly^n\cdot\id_{\mathscr{V}(\mathbf{\Sigma})}. \]
\end{enumerate}
\end{definition}

Now we may formulate the main result of the paper.

\begin{theorem} \label{thm_tqft}
There exists a unique (up to isomorphism) Topological Quantum Field Theory $(\mathscr{V},Z)$ satisfying the following conditions:
\begin{enumerate}

\item \label{con_theta}
$\mathscr{V}$ assigns the space of theta functions $\thetafn(\surf)\cong\rlskein{\hbdy}$ to the standard extended surface $\esurf:=(\surf,\mathbf{L}_g)$:
\[ \mathscr{V}(\esurf) = \rlskein{\hbdy}. \]

\item \label{con_mcgact}
The representation
\begin{displaymath}
\begin{array}{ccc}
\Aut(\esurf) & \to & \End\left(\rlskein{\hbdy}\right), \\
\mathbf{h} & \mapsto & \mathscr{V}(\mathbf{h});
\end{array}
\end{displaymath}
of the group of automorphisms of a standard extended surface on the space of states coincides with the representation defined by \eqref{eqn_masgrpact}:
\[ \mathscr{V}(\mathbf{h}) = \mathcal{F}(\mathbf{h}), \quad\text{for all }\mathbf{h}\in\Aut(\esurf). \]

\item \label{con_hbody}
Given a link $L$ inside the standard handlebody $\hbdy$, consider the framed 3-cobordism
\[ \mathbf{H}_g[L]:=\left(\hbdy,\mathbf{\Phi},\esurf,L,0\right). \]
We require that
\[ Z(\mathbf{H}_g[L]) = \skn{L}\in\rlskein{\hbdy} = \mathscr{V}(\esurf). \]

\item \label{con_linknum}
Given a link $L$ inside $S^3$, consider the closed framed 3-cobordism
\[ \mathbf{S}^3[L]:=\left(S^3,\mathbf{\Phi},\mathbf{\Phi},L,0\right). \]
We require that there exists a constant $\kappa\in\mathbb{C}$ such that for all links $L$ in $S^3$,
\[ Z\left(\mathbf{S}^3[L]\right) = \kappa \skn{L}\in\rlskein{S^3}=\mathbb{C}. \]
\end{enumerate}
The value of  $\kappa$ is uniquely determined by the above conditions; it is given by $\kappa=N^{-\frac{1}{2}}$. The anomaly of this TQFT, also completely determined by the above, is given by $\anomaly=e^{-\frac{\pi i}{4}}$.
\end{theorem}

\begin{remark}
Condition \eqref{con_theta} states the obvious requirement that the space of states for our theory be the space of theta functions. Condition \eqref{con_mcgact} expresses the most significant requirement, that the representation of the mapping class group that is associated to the theory must be the standard action of the mapping class group on theta functions. This action is determined (up to a constant) by the principles of quantization. Conditions \eqref{con_hbody} and \eqref{con_linknum} are natural requirements on the cobordism theory laid out in Section \ref{sec_framedbord}. They state that this picture of cobordisms with embedded links must be consistent with the manner in which we modeled the space of theta functions in Section \ref{sec_skeinthy} using skein theory.

Condition \eqref{con_hbody} also contains the information about the Schr\"odinger representation \eqref{eqn_schrorep}, as modeled using skein theory by Theorem \ref{thm_opskein} and Theorem \ref{thm_thetaskein}. Consider a framed oriented link $L$ embedded in the cylinder $\surfcyl$ and the corresponding cobordism
\[ \left(\surfcyl,(\surf\times\{0\},\mathbf{L}_g),(\surf\times\{1\},\mathbf{L}_g),L,0\right) \]
formed by this cylinder with this embedded link. Condition \eqref{con_hbody}, together with Axiom \eqref{axiom_gluing} of Definition \ref{def_tqft}, ensures that the map assigned by the TQFT to this cobordism coincides with the action of the skein $\skn{L}$ in $\rlskein{\surfcyl}$ on $\rlskein{\hbdy}$. In this way, using theorems \ref{thm_opskein} and \ref{thm_thetaskein}, we see that the TQFT incorporates the action of the finite Heisenberg group $\heis(\mathbb{Z}_N^g)$ on the space of theta functions $\thetafn(\surf)$ through the maps assigned to these cobordisms.
\end{remark}

\begin{proof}
We divide the proof of the theorem into three parts. We first construct a preliminary TQFT on the parameterized 3-cobordisms of Section \ref{sec_surgerycobo} with embedded links. Next, we show that a TQFT on framed 3-cobordisms exists satisfying all the requirements of Definition \ref{def_tqft} and the conditions above. We then explain how this description of the TQFT is forced upon us by conditions (1)--(4) of the theorem.

\subsubsection*{Construction of the TQFT: parameterized 3-cobordisms}

Let $M:\partial_-M\Rightarrow\partial_+M$ be a parameterized 3-cobordism with an embedded oriented framed link $L$. We define a map
\begin{equation} \label{eqn_zee}
Z(M,L):\left(\bigotimes_{i=1}^{k_-}\rlskein{\hbdy[g_i^-]}\right)\to\left(\bigotimes_{i=1}^{k_+}\rlskein{\hbdy[g_i^+]}\right),
\end{equation}
where the $g^{\pm}_i$ are the genera of the parameterizing surfaces given by \eqref{eqn_param}. We consider first the case of a connected cobordism. Present the parameterized 3-cobordism by a surgery diagram as in Section \ref{sec_surgerycobo}. Denote the surgery link of this diagram by $L_{\mathrm{surg}}$. Next, add the (nonsurgery) link $L$ embedded in $M$ to the surgery diagram. Consider the complement $C\subset S^3$ of a regular neighborhood of the ribbon graphs of the surgery diagram. We may define the skein
\[ \skn{L\cup \Omega(L_{\mathrm{surg}})}\in\rlskein{C}. \]
An example is shown in Figure~\ref{fig_skeinofcobo}.
\begin{figure}[h]
\centerline{\psfig{file=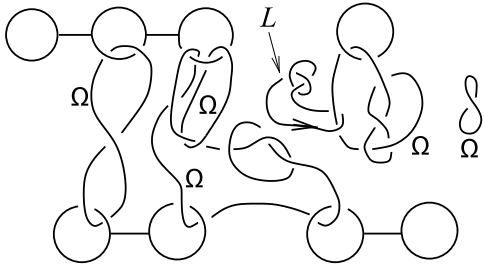,width=2.5in}}
\caption{The skein in $C$ associated to a parameterized framed 3-cobordism with embedded link $L$.}
\label{fig_skeinofcobo}
\end{figure}

This skein defines a bilinear map
\[ B_{M,L}:\left(\bigotimes_{i=1}^{k_-}\rlskein{\hbdy[g_i^-]}\right)\otimes\left(\bigotimes_{i=1}^{k_+}\rlskein{\hbdy[g_i^+]}\right) \to \rlskein{S^3}=\mathbb{C} \]
as follows. Embed the handlebodies $\hbdy[g_i^{\pm}]$ inside regular neighborhoods of the ribbon graphs of the surgery diagram using translations, with the handlebodies $\hbdy[g^-_i]$ mapping to the bottom of the surgery diagram and the handlebodies $\hbdy[g^+_i]$ mapping to the top. These embeddings provide a way to map the corresponding skeins into the complement of $C$, hence we may define $B_{M,L}$ by
\begin{equation} \label{eqn_invform}
B_{M,L}(a_-,a_+):= N^{-\frac{1}{2}}e^{-\frac{\pi i}{4}\sigma(L_{\mathrm{surg}})}\skn{a_-\cup L\cup \Omega(L_{\mathrm{surg}})\cup a_+}\in\rlskein{S^3}.
\end{equation}
The usual Kirby calculus arguments (cf. Remark \ref{rem_kirby}) and some simple calculations show that \eqref{eqn_invform} does not depend on the choice of surgery presentation for the parameterized 3-cobordism.

Let $\innprod$ denote the symmetric nondegenerate bilinear form on $\bigotimes_{i=1}^{k_+} \rlskein{\hbdy[g^+_i]}$ that is formed by taking the tensor product of the forms \eqref{eqn_skeinform} defined on each factor $\rlskein{\hbdy[g^+_i]}$. We may define $Z(M,L)$ by
\[ \langle Z(M,L)[a_-],a_+ \rangle = N^{\frac{k_+}{2}} B_{M,L}(a_-,a_+); \quad a_{\pm}\in\otimes_{i=1}^{k_{\pm}}\rlskein{\hbdy[g_i^{\pm}]}. \]

If $M$ is the disjoint union of connected parameterized 3-cobordisms, we just define \eqref{eqn_zee} by taking the tensor product of the maps defined for each connected component.

Given a homeomorphism $h:\surf\to\surf$, define the mapping cylinder $\mapcyl{h}$ of $h$ to be the parameterized 3-cobordism $\surfcyl$, where the top surface $\surf\times\{1\}$ is parameterized by the identity and the bottom surface $\surf\times\{0\}$ is parameterized by $h$. As explained in the course of the proof of Theorem \ref{thm_masgrpact}, we have
\[ B_{\mapcyl{h},\emptyset}(a_-,a_+) = N^{-\frac{1}{2}}\langle \mathcal{F}(h,0)[a_-],a_+\rangle; \quad a_{\pm}\in\rlskein{\hbdy}. \]
Consequently,
\begin{equation} \label{eqn_mcylact}
Z(\mapcyl{h},\emptyset)=\mathcal{F}(h,0).
\end{equation}

Let $M:\partial_- M \to \partial_+ M$ and $M':\partial_- M' \to \partial_+ M'$ be composable parameterized 3-cobordisms with embedded links $L$, $L'$ and let $M''$ be their composition \eqref{eqn_composition}. We claim that
\begin{equation} \label{eqn_assoc}
Z(M'',L\cup L') = e^{\frac{\tau\pi i}{4}} Z(M',L') Z(M,L),
\end{equation}
where
\[ \tau:=\tau\left((f'_-f_+^{-1})_*N_M(f_-)_*\left(\oplus_i\mathbf{L}_{g^-_i}\right),(f'_-)_*\left(\oplus_i\mathbf{L}_{g'^{-}_i}\right),N^{M'}(f'_+)_*\left(\oplus_i\mathbf{L}_{g'^+_i}\right)\right) \]
and $f_{\pm}$, $f'_{\pm}$ are the parameterizations \eqref{eqn_param} of the boundary surfaces.

By a standard argument that involves replacing disjoint unions with connected sums (see for instance Section IV.2.8 of \cite{turaev}), or by simply composing one component at a time, we may assume that our cobordisms are connected. Let us choose bases
\begin{displaymath}
\begin{array}{rl}
a_1^{k_1} a_2^{k_2} \cdots a_m^{k_m} \in \otimes_{i=1}^{k_-}\rlskein{\hbdy[g_i^-]}; & \quad k_1,k_2,\ldots,k_m\in\mathbb{Z}_N; \\
a_1^{j_1} a_2^{j_2} \cdots a_n^{j_n} \in \otimes_{i=1}^{k_+}\rlskein{\hbdy[g_i^+]}; & \quad j_1,j_2,\ldots,j_n\in\mathbb{Z}_N; \\
a_1^{l_1} a_2^{l_2} \cdots a_p^{l_p} \in \otimes_{i=1}^{k'_+}\rlskein{\hbdy[{g'_i}^+]}; & \quad l_1,l_2,\ldots,l_m\in\mathbb{Z}_N; \\
\end{array}
\end{displaymath}
for our skein modules consisting of curves $a_i$ of the form shown in Figure \ref{fig_surface}, where $m$, $n$ and $p$ denote the sum of the genera of the components of $\partial_- M$, $\partial_+ M$ and $\partial_+ M'$ respectively. By Equation \eqref{eqn_dualbasis}, proving the identity \eqref{eqn_assoc} is equivalent to showing that
\begin{multline} \label{eqn_assocsum}
e^{-\frac{\tau\pi i}{4}}N^{\frac{n}{2}}\left\langle Z(M'',L'\cup L)[a_1^{k_1}\cdots a_m^{k_m}],a_1^{l_1}\cdots a_p^{l_p} \right\rangle = \\
\sum_{j_1,\ldots,j_n} \left\langle Z(M',L')[a_1^{j_1}\cdots a_n^{j_n}],a_1^{l_1}\cdots a_p^{l_p} \right\rangle \left\langle Z(M,L)[a_1^{k_1}\cdots a_m^{k_m}],a_1^{j_1}\cdots a_n^{j_n} \right\rangle.
\end{multline}

Present the cobordisms $M$ and $M'$ by surgery diagrams $(L_{\mathrm{surg}},\Gamma_-,\Gamma_+)$ and $(L'_{\mathrm{surg}},\Gamma'_-,\Gamma'_+)$ respectively. Recall from Theorem \ref{thm_turaevglue} that the parameterized 3-cobordism $M''$ is presented by the surgery diagram $(L'_{\mathrm{surg}}\circ L_{\mathrm{surg}},\Gamma_-,\Gamma'_+)$. We begin by computing the right-hand side of \eqref{eqn_assocsum}. This computation, with the notation
\[ u:=(k_+ + k'_+ -2)/2, \quad \upsilon:=\sigma(L'_{\mathrm{surg}})+\sigma(L_{\mathrm{surg}}), \quad\text{and}\quad w:=u+n/2; \]
is outlined in Figure~\ref{fig_cobcomp}.
\begin{figure}[h]
\includegraphics{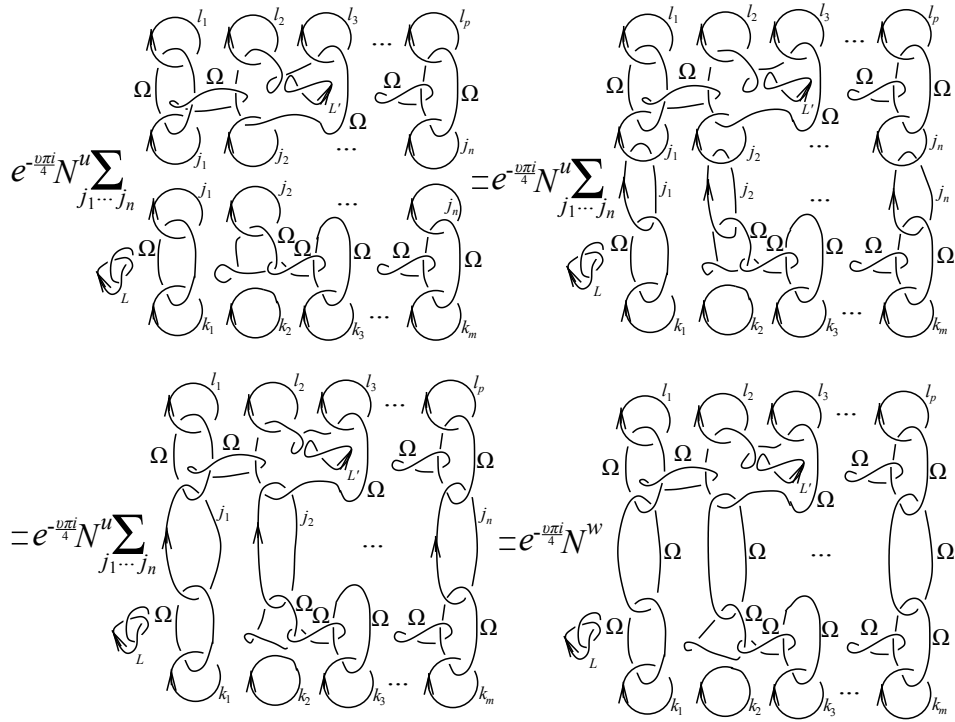}
\caption{Computation of right-hand side of \eqref{eqn_assocsum}.}
\label{fig_cobcomp}
\end{figure}

We are not yet done, as the surgery link in the final diagram of Figure~\ref{fig_cobcomp} is not quite $L'_{\mathrm{surg}}\circ L_{\mathrm{surg}}$. According to Theorem \ref{thm_turaevglue}, we must add $k_+-1$ circles around the groups of vertical lines in the middle of the diagram, and decorate them with the skein $\Omega$. In Figure~\ref{fig_addomega} we show that adding these circles amounts to introducing a factor of precisely $N^{(k_+-1)/2}$. This shows that the exponents of $N$ on both sides of \eqref{eqn_assocsum} agree, along with the skeins, so it only remains to check the powers of $e^{\frac{\pi i}{4}}$. This is verified using Wall's formula \cite{wall} for the signature and Remark \ref{rem_turaevglue}, which implies that
\[ \sigma(L'_{\mathrm{surg}}\circ L_{\mathrm{surg}}) + \tau = \sigma(L'_{\mathrm{surg}}) + \sigma(L_{\mathrm{surg}}). \]
This finishes the verification of Equation \eqref{eqn_assoc}.
\begin{figure}[h]
\centerline{\psfig{file=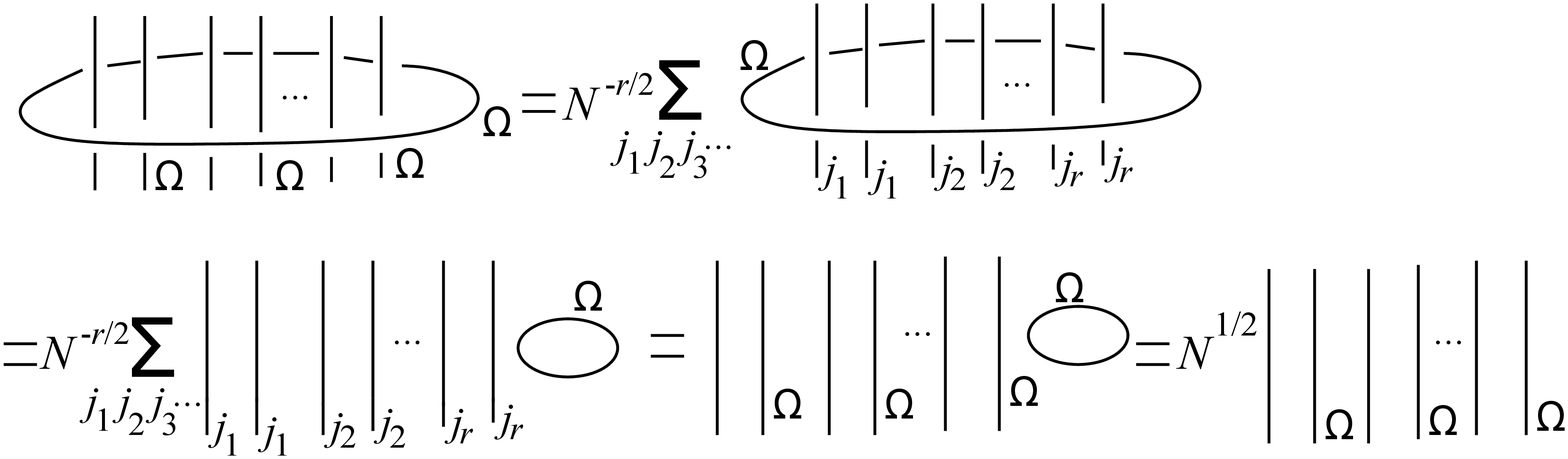,width=4.5in}}
\caption{Local computation in verifying Equation \eqref{eqn_assocsum}.}
\label{fig_addomega}
\end{figure}

\subsubsection*{Construction of the TQFT: framed 3-cobordisms}

We are now ready to define the TQFT whose existence is asserted by the theorem. Let $\mathscr{V}$ be the functor defined by Definition \ref{def_states}. By its construction, $\mathscr{V}$ clearly satisfies conditions \eqref{con_theta} and \eqref{con_mcgact} of the theorem, since we may canonically identify $\mathscr{V}(\esurf)$ with $\rlskein{\hbdy}$ using the identity parameterization of $\esurf$.

Let $\mathbf{M}:\partial_-\mathbf{M}\Rightarrow\partial_+\mathbf{M}$ be a framed 3-cobordism. Again, we discuss first the situation when $M$ is connected. To define a map
\begin{equation} \label{eqn_zeeframe}
Z(\mathbf{M}):\mathscr{V}(\partial_-\mathbf{M})\to\mathscr{V}(\partial_+\mathbf{M}),
\end{equation}
it is necessary, by the definition of $\mathscr{V}$ as a limit, to define a commuting system of maps
\[ Z_{\mathbf{f}_-,\mathbf{f}_+}(\mathbf{M}):\mathscr{L}(\mathbf{f}_-)\to\mathscr{L}(\mathbf{f}_+); \quad \mathbf{f}_{\pm}\in\mathscr{P}(\partial_{\pm}\mathbf{M}). \]
Let $\mathbf{f}_{\pm}$ be parameterizations of $\partial_{\pm}\mathbf{M}$. This gives $M$ the structure of a parameterized 3-cobordism $M_{f_-,f_+}$, and \eqref{eqn_zee} defines a map
\begin{equation} \label{eqn_zeeparam}
Z_{\mathbf{f}_-,\mathbf{f}_+}(\mathbf{M}):=e^{-\frac{n\pi i}{4}}Z(M_{f_-,f_+},L_M):\mathscr{L}(\mathbf{f}_-)\to\mathscr{L}(\mathbf{f}_+);
\end{equation}
where
\[ n:= n_M+n_{f_-}-n_{f_+}+\tau\left(f_-^*(\oplus_{i}\mathbf{L}_{g^-_i}),\mathbf{L}_-,N^M(\mathbf{L}_+)\right) + \tau\left(N_M(f_-^*(\oplus_{i}\mathbf{L}_{g^-_i})),\mathbf{L}_+,f_+^*(\oplus_{i}\mathbf{L}_{g^+_i})\right). \]
To check that \eqref{eqn_zeeparam} defines a commuting system of maps, suppose that we change the parameterizations $\mathbf{f}_{\pm}$ of $\partial_{\pm}\mathbf{M}$ by automorphisms $\mathbf{h}_{\pm}$ of the parameterizing surface. Then the parameterized 3-cobordism $M_{f_-h_-,f_+h_+}$ is obtained from the parameterized 3-cobordism $M_{f_-,f_+}$ by gluing the mapping cylinders $I_{h_-}$ and $I_{h_+^{-1}}$ to the bottom and top of $M_{f_-,f_+}$ respectively. Hence by Equation \eqref{eqn_mcylact} and Equation \eqref{eqn_assoc},
\begin{equation} \label{eqn_comsys}
Z_{\mathbf{f}_-\mathbf{h}_-,\mathbf{f}_+\mathbf{h}_+}(\mathbf{M}) = \mathcal{F}(\mathbf{h}_+^{-1})Z_{\mathbf{f}_-,\mathbf{f}_+}(\mathbf{M})\mathcal{F}(\mathbf{h}_-)
\end{equation}
up to a power of $e^{-\frac{\pi i}{4}}$. To check that the powers of $e^{-\frac{\pi i}{4}}$ on both sides of \eqref{eqn_comsys} actually agree is in fact very simple, since by Remark \ref{rem_param} we may assume that
\begin{equation} \label{eqn_lagpres}
f_{\pm}^*(\oplus_{i}\mathbf{L}_{g^{\pm}_i})=\mathbf{L}_{\pm}.
\end{equation}
With this assumption, half of the Maslov indices in the calculation become zero, and the remaining half are seen to cancel. This establishes Equation \eqref{eqn_comsys} and yields a well-defined map \eqref{eqn_zeeframe}. It is now simple to check that conditions \eqref{con_hbody} and \eqref{con_linknum} of the theorem are a straightforward consequence of the preceding definitions.

For a disconnected cobordism $\mathbf{M}$, we define $Z(\mathbf{M})$ as a tensor product of the maps \eqref{eqn_zeeframe} over the connected components. This makes sense as $\mathscr{V}$ satisfies Axiom \eqref{axiom_permute} of Proposition \ref{prop_funcstates}. This ensures Axiom \eqref{axiom_cobunion} of Definition \ref{def_tqft} holds.

The remaining axioms of Definition \ref{def_tqft} now follow easily from the facts that have just been established. Axiom \eqref{axiom_gluing} follows from Equation \eqref{eqn_assoc}; to ensure the Maslov index calculation is trivial, choose parameterizations $\mathbf{f}_-$, $\mathbf{f}_+$ and $\mathbf{f}'_+$ satisfying \eqref{eqn_lagpres} and set $\mathbf{f}'_-:=\mathbf{h}\mathbf{f}_+$. To check Axiom \eqref{axiom_natural} is tautological, provided that we choose our parameterizations appropriately in the manner that has just been explained and which renders the Maslov index calculation trivial. The verification of Axiom \eqref{axiom_norm} is similarly tautological.

\subsubsection*{Proof of the uniqueness of the TQFT}

We now show how this description of the TQFT is forced upon us by conditions (1)--(4) of the theorem. Suppose that $(\mathscr{V}',Z')$ is another TQFT satisfying these conditions. Recall from Definition \ref{def_states} that $\mathscr{V}$ was defined as a limit. Hence, since $\mathscr{V}'$ satisfies conditions \eqref{con_theta} and \eqref{con_mcgact} of the theorem, we may define a natural equivalence between $\mathscr{V}$ and $\mathscr{V}'$ such that the following diagram commutes for all extended surfaces $\mathbf{\Sigma}$ and parameterizations $\mathbf{f}\in\mathscr{P}(\mathbf{\Sigma})$:
\[ \xymatrix{ \mathscr{V}(\mathbf{\Sigma}) \ar[r]^{\cong} & \mathscr{V}'(\mathbf{\Sigma}) \\ \mathscr{L}(\mathbf{f}) \ar[u]^{\iota_{\mathbf{f}}} \ar[ru]_{\mathscr{V}'(\mathbf{f})} } \]

Note that since $\mathscr{V}'$ satisfies condition \eqref{con_mcgact} we have
\[ a\cdot\id_{\mathscr{V}'(\esurf)} = \mathscr{V}'(\id_{\surf},1) = \mathcal{F}(\id_{\surf},1)=e^{-\frac{\pi i}{4}}\cdot\id_{\mathscr{V}'(\esurf)}; \]
hence the anomaly of our TQFT $(\mathscr{V}',Z')$ must also be $a=e^{-\frac{\pi i}{4}}$.

By a well-known and routine argument (cf. for instance Theorem III.3.3 of \cite{turaev}), to prove these two TQFTs are isomorphic, it suffices to check that they produce the same closed 3-manifold invariants. Hence, let
\[ \mathbf{M}:=\left(M,\mathbf{\Phi},\mathbf{\Phi},L,n\right) \]
be a closed connected framed 3-cobordism. The closed 3-manifold $M$ may be obtained by performing surgery on $S^3$ along a framed link
\[ L_{\mathrm{surg}} = L_1\cup\ldots\cup L_k, \]
where each $L_i$ is a closed framed curve in $S^3$. Now let
\[ f_i:\hbdy[1]\to S^3, \quad i=1,\ldots, k; \]
be the orientation preserving embeddings of standard solid tori into $S^3$ that are determined by these framed curves, and denote their images in $S^3$ by $T_i:=f_i(\hbdy[1])$.

Remove from $S^3$ the interiors of the tori $T_i$ and define the following Lagrangian,
\[ \mathbf{L}:=\ker\left[H_1\left(\cup_{i=1}^k \partial T_i,\mathbb{Z}\right) \to H_1\left(S^3-\Int(\cup_{i=1}^k T_i),\mathbb{Z}\right)\right]. \]
Consider the framed 3-cobordism
\[ (\mathbf{S\setminus T}):=\left(S^3-\Int(\cup_{i=1}^k T_i),(\cup_{i=1}^k \partial T_i,\mathbf{L}),\mathbf{\Phi},L,0\right) \]
which contains the framed link $L$ originating in $\mathbf{M}$. Let $\varphi$ be the homeomorphism of the standard torus that rotates the meridian and longitude;
\[ \varphi:\surf[1]\to\surf[1]; \qquad a\mapsto b^{-1}, \quad b\mapsto a; \]
where $a:=a_1$ and $b:=b_1$ are the curves on the torus shown in Figure \ref{fig_surface}. Then $\mathbf{M}$ is obtained by gluing the framed 3-cobordism $(\mathbf{S\setminus T})$ to the framed 3-cobordism
\[ \mathbf{H}:=\bigsqcup_{i=1}^k(\hbdy[1],\mathbf{\Phi},\esurf[1],\emptyset,0) \]
along the extended homeomorphism
\[ \left(\sqcup_{i=1}^k[{f_i}_{|\surf[1]}\circ\varphi],n\right) = \left(\sqcup_{i=1}^k {f_i}_{|\surf[1]},0\right)\circ (\id,n-\tau)\circ \left(\sqcup_{i=1}^k (\varphi,0)\right), \]
where $\tau:=\tau\left(\oplus_{i=1}^k\left[{f^*_i}_{|\surf[1]}(\mathbf{A})\right],\oplus_{i=1}^k\left[{f^*_i}_{|\surf[1]}(\mathbf{B})\right],\mathbf{L}\right)$ and $\mathbf{A}$ and $\mathbf{B}$ are the Lagrangians in $H_1(\surf[1],\mathbb{Z})$ generated by the curves $a$ and $b$ respectively.

A fairly routine calculation, which we do not reproduce, shows that
\begin{equation} \label{eqn_ftrans}
\mathcal{F}(\varphi,0)\cdot\skn{\emptyset} = N^{-\frac{1}{2}}\sum_{j\in\mathbb{Z}_N} a^j,
\end{equation}
where $\skn{\emptyset}\in\rlskein{\hbdy[1]}$ denotes the empty skein. We are now ready to calculate the invariant $Z'(\mathbf{M})$ and show that it is equal to $Z(\mathbf{M})$:
\begin{equation} \label{eqn_unique}
\begin{split}
Z'(\mathbf{M}) &= Z'(\mathbf{S\setminus T})\circ\mathscr{V}'\left(\sqcup_{i=1}^k[{f_i}_{|\surf[1]}\circ\varphi],n\right)\circ Z'(\mathbf{H}), \\
&= e^{-\frac{(n-\tau)\pi i}{4}}Z'(\mathbf{S\setminus T})\circ\mathscr{V}'\left(\sqcup_{i=1}^k{f_i}_{|\surf[1]},0\right)\circ \left(\otimes_{i=1}^k \left[\mathcal{F}(\varphi,0)Z'(\mathbf{H}_1[\emptyset])\right]\right), \\
&= N^{-\frac{k}{2}}e^{-\frac{(n-\tau)\pi i}{4}}Z'(\mathbf{S\setminus T})\circ\mathscr{V}'\left(\sqcup_{i=1}^k{f_i}_{|\surf[1]},0\right)\left[\left(\sum_{j\in\mathbb{Z}_N} a^j\right)^{\otimes k}\right], \\
&= N^{-\frac{k}{2}}e^{-\frac{(n-\tau)\pi i}{4}}\sum_{j_1,\ldots,j_k\in\mathbb{Z}_N}Z'(\mathbf{S\setminus T})\circ\mathscr{V}'\left(\sqcup_{i=1}^k{f_i}_{|\surf[1]},0\right)\circ Z'\left(\sqcup_{i=1}^k \mathbf{H}_1[a^{j_i}]\right), \\
&= N^{-\frac{k}{2}}e^{-\frac{(n-\tau)\pi i}{4}}\sum_{j_1,\ldots,j_k\in\mathbb{Z}_N}Z'\left(\mathbf{S}^3\left[L\cup L_1^{\| j_1}\cup\ldots\cup L_k^{\| j_k}\right]\right), \\
&= \kappa N^{-\frac{k}{2}}e^{-\frac{(n-\tau)\pi i}{4}}\sum_{j_1,\ldots,j_k\in\mathbb{Z}_N}\skn{L\cup L_1^{\| j_1}\cup\ldots\cup L_k^{\| j_k}}, \\
&= \kappa e^{-\frac{(n-\tau)\pi i}{4}}\skn{L\cup\Omega(L_{\mathrm{surg}})}.
\end{split}
\end{equation}
The first line follows from Axiom \eqref{axiom_gluing} of Definition \ref{def_tqft}. The second line follows from Axiom \eqref{axiom_cobunion} of Definition \ref{def_tqft} and condition \eqref{con_mcgact} of the theorem. The third line follows from Equation \eqref{eqn_ftrans} and condition \eqref{con_hbody}. On the fourth line we have used Axiom \eqref{axiom_cobunion} of Definition \ref{def_tqft} and condition \eqref{con_hbody} of the theorem. The fifth line follows from Axiom \eqref{axiom_gluing} of Definition \ref{def_tqft}. The sixth line follows from condition \eqref{con_linknum} of the theorem.

Consider the 4-manifold that is obtained by gluing 2-handles to the 4-ball along the framed curves $L_i$. The signature of this 4-manifold, which coincides with that of the link $L_{\mathrm{surg}}$, is given by Wall's formula \cite{wall} to be $\sigma(L_{\mathrm{surg}})=-\tau$.

It remains to determine the value of the constant $\kappa$. By a standard argument (cf. for instance Theorem III.2.1.3 of \cite{turaev}) that involves gluing two cylinders over the 2-sphere together, we must have
\begin{equation} \label{eqn_sphcyl}
Z'(S^2\times S^1,\mathbf{\Phi},\mathbf{\Phi},\emptyset,0) = \Dim(\mathscr{V}'(S^2)) = \Dim(\rlskein{\hbdy[0]}) = 1.
\end{equation}
However, $S^2\times S^1$ is obtained from $S^3$ by surgery on the trivial framed knot, so by the formula we arrived at in \eqref{eqn_unique}, \eqref{eqn_sphcyl} must be equal to $\kappa\sqrt{N}$. This shows that $\kappa=N^{-\frac{1}{2}}$ and hence that our two TQFTs have the same closed 3-manifold invariants. This concludes the proof of the theorem.
\end{proof}

\begin{remark}
A legitimate question is whether the preceding TQFT can be constructed from a modular tensor category using the results of Reshetikhin-Turaev \cite{reshetikhinturaev}, \cite{turaev}. Unfortunately, any modular category producing this TQFT must possess a somewhat pathological property. One can prove that if $\mathscr{C}$ is a modular tensor category that gives rise to a TQFT whose space of states may be identified with theta functions, as in condition \eqref{con_theta} of Theorem \ref{thm_tqft}, and if $N$ is the product of distinct primes, then there is a simple object $V$ of $\mathscr{C}$, which generates all the simple objects of $\mathscr{C}$, such that $V^{\otimes N}$ is isomorphic (by a morphism $f$ say) to the unit object $\mathbbm{1}$ of $\mathscr{C}$, and such that the following diagram commutes:
\[ \xymatrix{ V^{\otimes N}\otimes V \ar[rr]^{f\otimes\id}_{\cong} \ar@{=}[d] && \mathbbm{1}\otimes V \ar@{=}[r] & V \ar[d]_{-\id} \\ V\otimes V^{\otimes N} \ar[rr]^{\id\otimes f}_{\cong} && V\otimes\mathbbm{1} \ar@{=}[r] & V } \]

This implies that $\mathscr{C}$ cannot be embedded in the tensor category of vector spaces, at least with the usual associator for the tensor product. This does not rule out a modular tensor category based on vector spaces with more exotic formulae for the associators, such as those considered in \cite{kazwenzl} and \cite{sterling} in which one multiplies by a power of a root of unity that depends upon which irreducible representation one finds oneself in. But, even if one could find such a modular category, we feel these categories are somewhat unpleasant to manipulate, and find the prospect of working in them unappealing when more convenient descriptions are available.
\end{remark}


\begin{thebibliography}{000}
%
\bibitem{andersen}
J. E. Andersen, \emph{Deformation quantization and geometric quantization of abelian moduli spaces.} Comm. Math. Phys., {\bf 255} (2005), 727--745.
%
\bibitem{bhmv}
C. Blanchet, N. Habegger, G. Masbaum, P. Vogel, \emph{Topological quantum field theories derived from the Kauffman bracket.} Topology, {\bf 34} 1995, 883--927.
%
\bibitem{farkaskra}
H. M. Farkas, I. Kra; \emph{Riemann Surfaces.} Springer, second ed. 1991.
%
\bibitem{thetaTQFT}
R. Gelca, A. Uribe; \emph{From classical theta functions to topological quantum field theory.} in I. Gitler, E. Lupercio, {H.G}. Comp\'ean,  F. Turrubiates, eds., \emph{The influence of Solomon Lefschetz in geometry and topology: 50 years of {M}athematics at {C}investav}, Contemporary Math., Amer. Math. Soc., 2014.
%
\bibitem{kashiwarashapira}
M. Kashiwara,  P. Schapira, \emph{Sheaves on Manifolds.} Springer, 1990.
%
\bibitem{kazwenzl}
D. Kazhdan, H. Wenzl; \emph{Reconstructing monoidal categories.} I. M. Gel'fand Seminar, 111--136, Adv. Soviet Math., 16, Part 2, Amer. Math. Soc., Providence, RI, 1993.
%
\bibitem{lionvergne}
G. Lion, M. Vergne; \emph{The Weil representation, Maslov index and theta series.} Progress in Mathematics, 6. Birkhäuser, Boston, Mass., 1980.
%
\bibitem{manin1}
Yu. I. Manin, \emph{Quantized theta functions.} Progress of Theor. Phys. Suppl., {\bf 102} 1990.
%
\bibitem{mum}
D. Mumford, Tata Lectures on Theta, Birkhauser, 1983.
%
\bibitem{moo}
H. Murakami, T. Ohtsuki, M. Okada, \emph{Invariants of three-manifolds derived from linking matrices of framed links.} Osaka J. Math., {\bf 29} 1992, 545--572.
%
\bibitem{przytycki2}
J. H. Przytycki, \emph{A q-analogue of the first homology group of a 3-manifold.} L. A. Coburn, M. A. Rieffel, eds., \emph{Perspectives on Quantization (Proceedings of the joint AMS-IMS-SIAM Conference on Quantization, Mount Holyoke College, 1996)}, Contemporary Math. 214, Amer. Math. Soc, 1998, 135--144.
%
\bibitem{reshetikhinturaev}
N. Reshetikhin, V. Turaev, \emph{Invariants of $3$-manifolds via link polynomials and quantum groups.} Invent. Math., \textbf{103} (1991), 547--597.
%
\bibitem{riemann}
B. Riemann, \emph{Theorie der Abel'schen Funktionen.} J. Reine und Angew. Math., {\bf 54} 1857, 101--155.
%
\bibitem{sniatycki}
J. \'{S}niatycki, \emph{Geometric Quantization and Quantum Mechanics.} Springer, 1980.
%
\bibitem{sterling}
S. Sterling, \emph{Abelian Chern-Simons Theory with Toral Gauge Group, Modular Tensor Categories and Group Categories.} ProQuest UMI Dissertation Publ., 2011.
%
\bibitem{turaev}
V. G. Turaev, \emph{Quantum Invariants of Knots and 3-Manifolds.} de Gruyter Studies in Mathematics, de Gruyter, Berlin--New York, 1994.
%
\bibitem{walker}
K. Walker, \emph{On {W}itten's 3-manifold invariants.} Preprint, 1991, \texttt{http://canyon23.net/math/1991TQFTNotes.pdf}.
%
\bibitem{wall}
C. T. C. Wall, \emph{Non-additivity of the signature.} Invent. Math. {\bf 7} 1969 269--274.
%
\bibitem{weil}
A. Weil, \emph{Sur certains groupes d'operateurs unitaires.} Acta Math., {\bf 111} 1964, 143--211.
%
\bibitem{woodhouse}
N. M. J. Woodhouse, \emph{Geometric quantization.} Oxford University Press, second ed. 1997.

\end{thebibliography}
\end{document}